\newif\ifJOURNAL
\JOURNALfalse
\newif\ifWP
\WPfalse
\newif\ifBASIC
\BASICfalse

\newif\ifFULL
\FULLfalse
\newif\ifLATIN
\LATINfalse

\WPtrue

\LATINtrue	

\newif\ifnotJOURNAL		
\notJOURNALtrue
\ifJOURNAL\notJOURNALfalse\fi

\newif\ifnotWP		
\notWPtrue
\ifWP\notWPfalse\fi

\newif\ifnotFULL	
\notFULLtrue
\ifFULL\notFULLfalse\fi

\newif\ifnotLATIN	
\notLATINtrue
\ifLATIN\notLATINfalse\fi

\newcommand{\Takemura}{takemura/etal:zero-one}
\ifnotLATIN
  \newcommand{\Grundbegriffe}{kolmogorov:1933full}
  \newcommand{\Cornfeld}{cornfeld/etal:1982}
  
\fi
\ifLATIN
  \newcommand{\Grundbegriffe}{kolmogorov:1933}
  \newcommand{\Cornfeld}{cornfeld/etal:1982latin}
  
\fi

\ifJOURNAL
\documentclass{article}
\usepackage{amsmath,amsfonts,amssymb,amsthm,latexsym,graphicx}
\newcommand{\Extra}[1]{}
\fi

\ifWP
\documentclass{article}
\usepackage{amsmath,amsfonts,amssymb,amsthm,latexsym,graphicx}

\makeatletter

\newif\iftwodates
\twodatesfalse

\renewcommand\maketitle{\begin{titlepage}%
  \let\footnotesize\small
  \let\footnoterule\relax
  \let \footnote \thanks
  \null\vfil
  \vskip 30\p@
  \begin{center}%
    {\LARGE \bf \@title \par}%
    \vskip 3em%
    {\large
     \lineskip .75em%
     \begin{tabular}[t]{c}%
       \@author
     \end{tabular}\par}%
     \vskip 1.5em%
  \end{center}\par
  \vfill
  \begin{center}
    \raisebox{1.5cm}{\includegraphics[width=0.58\textwidth]%
      {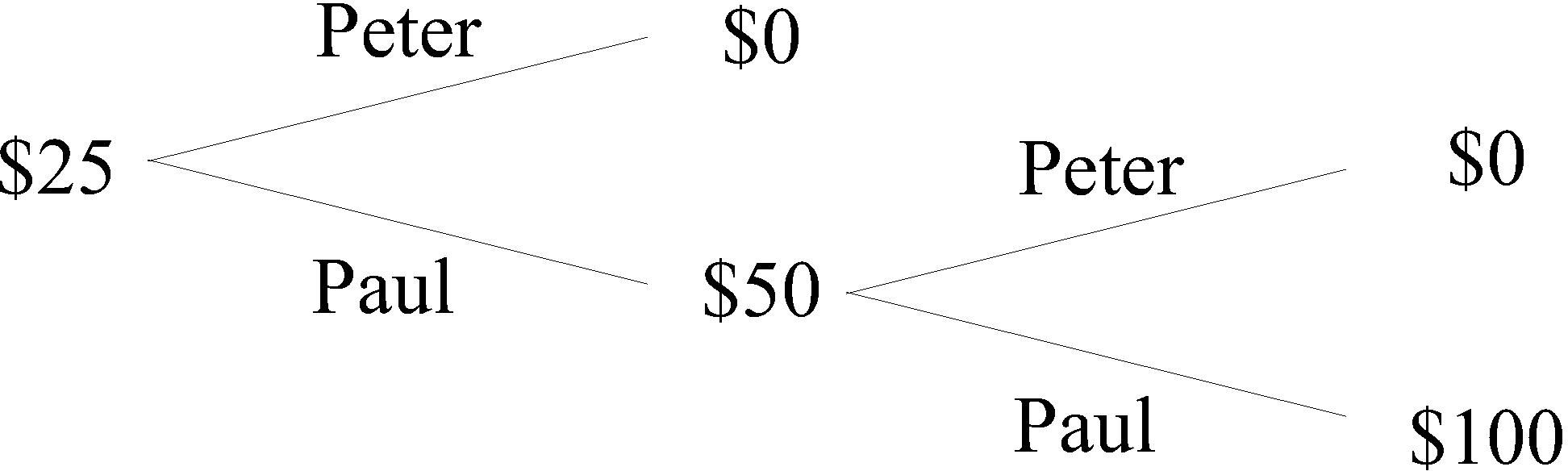}}%
    \hskip 3em%
    \includegraphics[width=0.29\textwidth]%
      {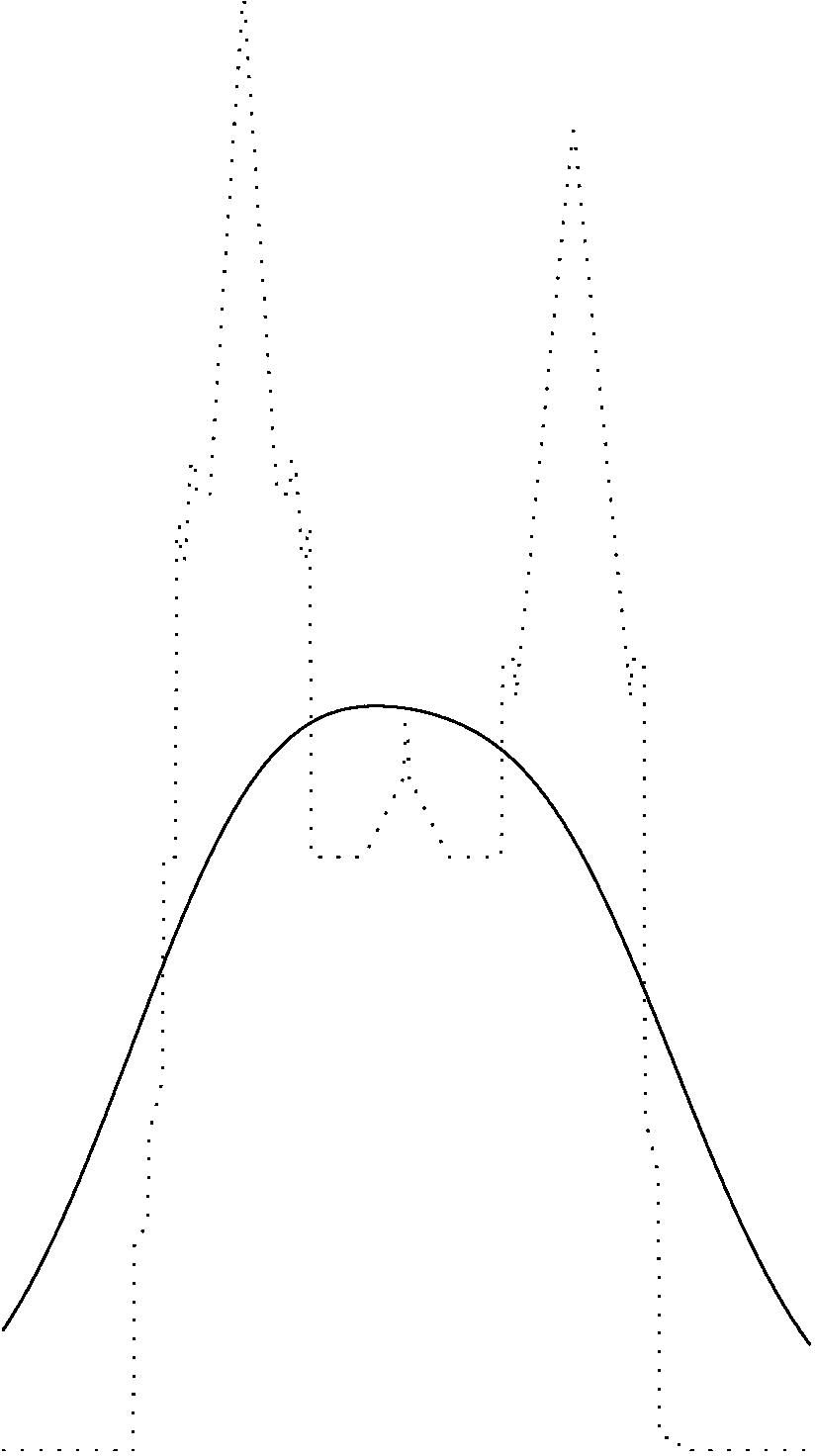}%
  \end{center}
  \@thanks
  \vfill
  \begin{center}
    {\large \bf The Game-Theoretic Probability and Finance Project}
  \end{center}
  \begin{center}
    {\large Working Paper \#\No}
  \end{center}
  \begin{center}
    {\iftwodates\large First posted \firstposted.
    Last revised \@date.\else\large\@date\fi}
  \end{center}
  \begin{center}
    Project web site:\\
    http://www.probabilityandfinance.com
  \end{center}
  \end{titlepage}%
  \setcounter{footnote}{0}%
  \global\let\thanks\relax
  \global\let\maketitle\relax
  \global\let\@thanks\@empty
  \global\let\@author\@empty
  \global\let\@date\@empty
  \global\let\@title\@empty
  \global\let\title\relax
  \global\let\author\relax
  \global\let\date\relax
  \global\let\and\relax
}

\renewenvironment{abstract}{%
  \titlepage
  \null\vfil
  \@beginparpenalty\@lowpenalty
  \begin{center}%
    \Large \bfseries \abstractname
    \@endparpenalty\@M
  \end{center}}%
  {\par\vfill\tableofcontents\endtitlepage}

\renewenvironment{thebibliography}[1]
  {\section*{\refname}%
  \addcontentsline{toc}{section}{\refname}
  \@mkboth{\MakeUppercase\refname}{\MakeUppercase\refname}%
  \list{\@biblabel{\@arabic\c@enumiv}}%
    {\settowidth\labelwidth{\@biblabel{#1}}%
    \leftmargin\labelwidth
    \advance\leftmargin\labelsep
    \@openbib@code
    \usecounter{enumiv}%
    \let\p@enumiv\@empty
    \renewcommand\theenumiv{\@arabic\c@enumiv}}%
    \sloppy
    \clubpenalty4000
    \@clubpenalty \clubpenalty
    \widowpenalty4000%
    \sfcode`\.\@m}
    {\def\@noitemerr
    {\@latex@warning{Empty `thebibliography' environment}}%
  \endlist}

\makeatother

\newcommand{\Extra}[1]{}
\fi

\ifBASIC
\documentclass[12pt]{article}
\usepackage{amsmath,amsfonts,amssymb,amsthm,latexsym,graphicx}
\newcommand{\Extra}[1]{}
\fi

\ifFULL
\usepackage{color}
\renewcommand{\Extra}[1]{\blue{#1}}

\newcommand{\blue}[1]{\textcolor{blue}{#1}}
\newcommand{\bluebegin}{\begingroup\color{blue}}
\newcommand{\blueend}{\endgroup}

\fi

\allowdisplaybreaks[4]

\newcommand{\Vladimir}{Vladimir}
\newcommand{\DOT}{.}

\ifnotLATIN
  \input{OT2enc.def}
  
  \usepackage{CJK}
\fi

\gdef\mathenddots{\mathinner{\ldotp\ldotp\ldotp\ldotp}}
\gdef\endsentence{\spacefactor=3000}

\newcommand{\st}{\mathrel{|}}
\newcommand{\givn}{\mathrel{|}}

\newcommand{\dd}{\mathrm{d}}		

\newcommand{\K}{\mathcal{K}}		

\newcommand{\EEE}{\mathcal{E}}		
\newcommand{\FFF}{\mathcal{F}}		
\newcommand{\PPP}{\mathcal{P}}		
\newcommand{\SSS}{\mathcal{S}}		
\newcommand{\TTT}{\mathcal{T}}		
\newcommand{\UUU}{\mathcal{U}}		
\newcommand{\XXX}{\mathcal{X}}		

\DeclareMathOperator{\III}{\mathbb{I}}		

\newcommand{\bbbn}{\mathbb{N}}			
\newcommand{\bbbr}{\mathbb{R}}			
\newcommand{\bbbrbar}{\overline{\mathbb{R}}}	
\newcommand{\bbbrbarX}{\smash{\bbbrbar}\vphantom{\bbbr}^{X}}
\newcommand{\bbbrbarXXX}{\smash{\bbbrbar}\vphantom{\bbbr}^{\XXX}}
\newcommand{\bbbrbarOmega}{\smash{\bbbrbar}\vphantom{\bbbr}^{\Omega}}

\newcommand{\bbbp}{\mathbb{P}}			

\DeclareMathOperator{\UpperProb}{\lefteqn{\smash{\overline{\bbbp}}}\phantom{\bbbp}}
\DeclareMathOperator{\LowerProb}{\underline{\bbbp}}	

\newcommand{\bbbe}{\mathbb{E}}			

\DeclareMathOperator{\UpperExpect}{\lefteqn{\smash{\overline{\bbbe}}}\phantom{\bbbe}}
\DeclareMathOperator{\LowerExpect}{\underline{\bbbe}}	

\newtheorem{theorem}{Theorem}

\newtheorem{lemma}{Lemma}
\newtheorem{corollary}{Corollary}
\theoremstyle{definition}
\newtheorem{example}{Example}
\newtheorem*{remark}{Remark}

\newlength{\IndentI}
\newlength{\IndentII}
\newlength{\IndentIII}
\newlength{\IndentIV}
\newlength{\WidthI}
\newlength{\WidthII}
\newlength{\WidthIII}
\newlength{\WidthIV}
\ifJOURNAL
  \title{L\'evy's zero-one law in game-theoretic probability}
  \author{Glenn Shafer$^{1,2}$\\
    {\small gshafer@andromeda.rutgers.edu, http://glennshafer.com}\\[1mm]
    Vladimir Vovk$^2$ (corresponding author)\\
    {\small vovk@cs.rhul.ac.uk, http://vovk.net}\\[1mm]
    Akimichi Takemura$^3$\\
    {\small takemura@stat.t.u-tokyo.ac.jp, http://park.itc.u-tokyo.ac.jp/atstat/athp/}}
\fi

\ifWP
  \title{L\'evy's zero-one law in game-theoretic probability}
  \author{Glenn Shafer, Vladimir Vovk, and Akimichi Takemura}
  \newcommand{\No}{29}
  \twodatestrue
  \newcommand{\firstposted}{May 3, 2009}
\fi

\ifBASIC
\title{L\'evy's zero-one law in game-theoretic probability}
\author{Glenn Shafer$^{1,2}$, Vladimir Vovk$^2$, and Akimichi Takemura$^3$}
\fi

\begin{document}

\ifnotJOURNAL
  \maketitle
\fi

\begin{abstract}
  We prove a non-stochastic version of L\'evy's zero-one law,
  and deduce several corollaries from it,
  including non-stochastic versions of Kolmogorov's zero-one
  law\ifnotJOURNAL, \fi
  the ergodicity of Bernoulli shifts\ifnotJOURNAL,
    and a zero-one law for dependent trials\fi.
  Our secondary goal is to explore the basic definitions
  of game-theoretic probability theory,
  with L\'evy's zero-one law serving a useful role.
\end{abstract}

\section{Introduction}

In this article,
we prove a game-theoretic version of L\'evy's zero-one law.
It applies in situations
where standard statements of L\'evy's zero-one law
(\cite{levy:1937}, Section~41) do not apply,
because we do not postulate a probability measure on outcomes.
This is typical for game-theoretic probability:
see, e.g., \cite{shafer/vovk:2001}.
Upper and lower probabilities do emerge naturally
in prediction protocols considered in game-theoretic probability,
but in many cases lower probabilities are strictly less
than the corresponding upper probabilities,
and so they fall short of defining a probability measure.


The investigation of zero-one laws of game-theoretic probability
was started in \cite{\Takemura}.
From our game-theoretic version of L\'evy's zero-one law
we deduce game-theoretic versions
of Kolmogorov's zero-one law
(\cite{\Grundbegriffe}, Appendix)\ifnotJOURNAL, \fi
the ergodicity of Bernoulli shifts
(see, e.g., \cite{\Cornfeld}, Section 8.1, Theorem 1)\ifnotJOURNAL,
  and B\'artfai and R\'ev\'esz's \cite{bartfai/revesz:1967} zero-one law\fi.
\ifnotJOURNAL The first two \fi
results have been established in \cite{\Takemura},
but our proofs are different:
we obtain them as easy corollaries of our main result.

We start our exposition in Section~\ref{sec:expectation}
by introducing our basic prediction protocol
and defining the game-theoretic notions of expectation and probability
(upper and lower);
these definitions are explored in Section~\ref{sec:equivalent}.
In Section \ref{sec:levy} we prove L\'evy's zero-one law
for this protocol.
In the second part of Section \ref{sec:levy}
we consider the special case of an event
whose upper probability coincides with its lower probability;
our version of L\'evy's zero-one law for such events
looks much more similar to the standard statement.
Section~\ref{sec:implication} describes
a useful application of L\'evy's zero-one law
to the foundations of game-theoretic probability theory.
In Section \ref{sec:explicit} we derive two other non-stochastic zero-one laws
(Kolmogorov's zero-one law and the ergodicity of Bernoulli shifts)
as corollaries.
In Section \ref{sec:generality} we explain that our prediction protocol
covers as special cases seemingly more general protocols
considered in literature.
\ifnotJOURNAL
  This helps us to derive a non-stochastic version of one more zero-one law
  (B\'artfai and R\'ev\'esz's)
  in Section~\ref{sec:bartfai-revesz}.
\fi

We will be using the standard notation
$\bbbn=\{1,2,\ldots\}$
for the set of all natural numbers
and $\bbbr=(-\infty,\infty)$ for the set of all real numbers.
Alongside $\bbbr$
we will often consider sets,
such as $(-\infty,\infty]$ and $\bbbrbar:=[-\infty,\infty]$,
obtained from $\bbbr$ by adding $-\infty$ or $\infty$ (or both).
We set $0\times\infty:=0$ and $\infty+(-\infty):=\infty$
(the operation $+$ extended to $\bbbrbar$ by $\infty+(-\infty):=\infty$
is denoted by $\stackrel{\ldotp}{+}$ in \cite{hoffmann-jorgensen:1987};
in this article we will omit the dot in $\stackrel{\ldotp}{+}$
since we do not need any other extensions of $+$).
The indicator function of a subset $E$ of a given set $X$
will be denoted $\III_E$;
i.e., $\III_E:X\to\bbbr$ takes the value $1$ on $E$
and the value $0$ outside $E$.
The words such as ``positive'' and ``negative''
are to be understood
in the wide sense of inequalities~$\ge$ and~$\le$
rather than~$>$ and~$<$.

\section{Game-theoretic expectation and probability}
\label{sec:expectation}

We consider a perfect-information game
between two players called World and Skeptic.
The game proceeds in discrete time.
First we describe the game formally,
and then briefly explain the intuition behind the formal description.

Let $X$ be a set,
and let $\bbbrbarX$ stand for the set of all functions $f:X\to\bbbrbar$.
A function $\EEE:\bbbrbarX\to\bbbrbar$
is called an \emph{outer probability content}
if it satisfies the following four axioms:
\begin{enumerate}
\item\label{ax:order}
  If $f,g\in\bbbrbarX$ satisfy $f\le g$,
  then $\EEE(f)\le\EEE(g)$.
\item\label{ax:scaling}
  If $f\in\bbbrbarX$ and $c\in(0,\infty)$, then $\EEE(cf)=c\EEE(f)$.
\item\label{ax:sum}
  If $f,g\in\bbbrbarX$, then $\EEE(f+g)\le\EEE(f)+\EEE(g)$.
\item\label{ax:norm}
  For each $c\in\bbbr$,
  $\EEE(c)=c$,
  where the $c$ in parentheses
  is the function in $\bbbrbarX$ that is identically equal to $c$.
\end{enumerate}
The function $\EEE$ is called a \emph{superexpectation functional}
(or \emph{superexpectation} for brevity)
if, in addition, it satisfies the following axiom
(sometimes referred to as \emph{$\sigma$-subadditivity on $[0,\infty]^X$}).
\begin{enumerate}
\setcounter{enumi}{4}
\item\label{ax:countable}
  For any sequence of positive functions $f_1,f_2,\ldots$ in $\bbbrbarX$,
  \begin{equation}\label{eq:subadditivity}
    \EEE
    \left(
      \sum_{k=1}^{\infty}
      f_k
    \right)
    \le
    \sum_{k=1}^{\infty}
    \EEE
    \left(
      f_k
    \right).
  \end{equation}
\end{enumerate}

Replacing the $=$ in Axiom~\ref{ax:scaling} with $\le$
leads to an equivalent statement
because, for $c\in(0,\infty)$,
$\EEE(cf)\le c\EEE(f)=c\EEE((1/c)cf)\le\EEE(cf)$.
In presence of Axiom~\ref{ax:order},
we can allow $c\in\bbbrbar$ in Axiom~\ref{ax:norm}
without changing the content of the latter.
Axiom~\ref{ax:norm} implies $\EEE(0)=0$
(so that we can allow $c=0$ in Axiom \ref{ax:scaling}).
This, in combination with Axiom \ref{ax:order},
implies
\begin{equation}\label{eq:weak-coherence}
  f\ge0
  \Longrightarrow
  \EEE(f)\ge0.
\end{equation}
Axioms \ref{ax:sum} and \ref{ax:norm} imply that
\begin{equation}\label{eq:constant}
  \EEE(f+c)=\EEE(f)+c
\end{equation}
for each $c\in\bbbr$
(indeed, $\EEE(f+c)\le\EEE(f)+\EEE(c)=\EEE(f)+c$
and $\EEE(f)\le\EEE(f+c)+\EEE(-c)=\EEE(f+c)-c$).
From (\ref{eq:weak-coherence}) and (\ref{eq:constant})
we can see that,
for any $c\in\bbbr$,
\begin{equation}\label{eq:coherence}
  \EEE(f)< c
  \Longrightarrow
  \inf_{x\in\XXX}f(x)< c.
\end{equation}

Axioms 1--5 are relaxations of the standard properties
of the expectation functional:
cf., e.g., Axioms 1--5 in \cite{whittle:2000}
(Axioms \ref{ax:scaling} and \ref{ax:sum}
are weaker than the corresponding standard axioms,
Axioms \ref{ax:order} and \ref{ax:norm}
are stronger than the corresponding standard axioms
but follow from standard Axioms~1--4,
and Axiom~\ref{ax:countable} follows from standard Axiom~5
in the presence of our Axiom~\ref{ax:sum}).

The most controversial axiom is Axiom~\ref{ax:countable}.
It is satisfied in many interesting cases,
such as in the case of finite $\XXX$
and for many protocols in \cite{shafer/vovk:2001}.
Axiom~\ref{ax:countable} is convenient and often makes proofs easier.
However, most of the results in this article hold without it,
as pointed out by a referee.
For our principal results,
we first prove them assuming Axiom~\ref{ax:countable}
but then give an additional argument to get rid of the reliance on it.

The most noticeable difference
between what we call superexpectation functionals
and the standard expectation functionals
is that the former are defined for all functions
$f:\XXX\to\bbbrbar$
whereas the latter are defined only for functions
that are measurable w.r.\ to a given $\sigma$-algebra.
The notion of superexpectation functional is more general
since every expectation functional can be extended
to the whole of $\bbbrbarX$ as the corresponding upper integral.
Namely, $\EEE(f)$ can be defined
as the infimum of the expectation of $g$
(taken to be $\infty$ whenever the expectation of $\max(g,0)$ is $\infty$)
over all measurable functions $g\ge f$.
The extension may no longer be an expectation functional
but is still a superexpectation functional.

\begin{remark}
  Superexpectation functionals have been studied in the past
  in many different contexts under different names;
  in our terminology we mainly follow \cite{hoffmann-jorgensen:1987},
  except that we abbreviate
  ``outer probability content $\sigma$-subadditive on $[0,\infty]^X$''
  to ``superexpectation functional''.
  Upper previsions studied in the theory of imprecise probabilities
  are closely related to superexpectation functionals,
  one difference being that upper previsions
  are only defined on the bounded functions in $\bbbrbarX$.
  There is a burgeoning literature, started by \cite{artzner/etal:1999},
  on coherent measures of risk,
  which are close to being mappings $f\mapsto\EEE(-f)$,
  where $\EEE$ is an outer probability content;
  coherent measures of risk, however,
  are usually defined only for functions $f$
  that are measurable and do not take values $\pm\infty$.
\end{remark}

\ifnotJOURNAL
\begin{remark}
  It is sometimes useful to have the stronger form
  \begin{equation}\label{eq:strong-coherence}
    f>0
    \Longrightarrow
    \EEE(f)>0
  \end{equation}
  of (\ref{eq:weak-coherence}).
  Even the strong form (\ref{eq:strong-coherence})
  follows from Axioms \ref{ax:order}--\ref{ax:countable}.
  Indeed, if $f>0$ but $\EEE(f)=0$,
  Axioms \ref{ax:order}--\ref{ax:countable} imply
  \begin{multline*}
    1
    \stackrel{\ref{ax:norm}}{=}
    \EEE
    \left(
      \III_{\{f>0\}}
    \right)
    =
    \EEE
    \left(
      \III_{\cup_{n=1}^{\infty}\{nf\ge1\}}
    \right)
    \stackrel{\ref{ax:order}}{\le}
    \EEE
    \left(
      \sum_{n=1}^{\infty}\III_{\{nf\ge1\}}
    \right)\\
    \stackrel{\ref{ax:countable}}{\le}
    \sum_{n=1}^{\infty}
    \EEE
    \left(
      \III_{\{nf\ge1\}}
    \right)
    \stackrel{\ref{ax:order}}{\le}
    \sum_{n=1}^{\infty}\EEE(nf)
    \stackrel{\ref{ax:scaling}}{=}
    0
  \end{multline*}
  (over each relation symbol
  we write the ordinal number of the axiom that justifies it;
  we could avoid using Axiom \ref{ax:scaling}
  by using (\ref{eq:weak-coherence}) and Axiom \ref{ax:sum} instead).
\end{remark}
\fi

The main prediction protocol that we consider in this article is as follows.

\medskip

\noindent
\textsc{Protocol~1. Basic prediction protocol}

\noindent
\textbf{Parameters:} non-empty set $\XXX$\\
  \hspace*{\IndentII}and outer probability contents $\EEE_1,\EEE_2,\ldots$ on $\XXX$

\noindent
\textbf{Protocol:}

\parshape=6
\IndentI   \WidthI
\IndentI   \WidthI
\IndentII  \WidthII
\IndentII  \WidthII
\IndentII  \WidthII
\IndentI   \WidthI
\noindent
Skeptic announces $\K_0\in\bbbrbar$.\\
FOR $n=1,2,\ldots$:\\
  Skeptic announces $f_n\in\bbbrbarXXX$ such that $\EEE_n(f_n)\le\K_{n-1}$.\\
  World announces $x_n\in\XXX$.\\
  $\K_n := f_n(x_n)$.\\
END FOR

\medskip

\noindent
The set $\XXX$ will be called the \emph{outcome space};
this set and the other parameters of the game
(namely, the outer probability contents $\EEE_1,\EEE_2,\ldots$ on $\XXX$)
will be fixed until Section~\ref{sec:generality}.
An important special case is
where $\EEE_1,\EEE_2,\ldots$ are superexpectations on $\XXX$,
but as we said, our principal results will not require this assumption.

At the beginning of each trial $n$
Skeptic chooses a gamble represented as a function $f_n$ on $\XXX$.
After that World chooses the outcome $x_n$ of this trial,
which determines the payoff $f_n(x_n)$ of Skeptic's gamble.
The gambles available to Skeptic at trial $n$ are determined by $\EEE_n$.
Skeptic's capital after the $n$th trial is denoted $\K_n$.
He is allowed to choose his initial capital $\K_0$
and, implicitly, also allowed to throw away part of his capital
at each trial.
Our definitions imply that Skeptic is allowed not to play
at trial $n$ (and thus keep his money intact)
by choosing $f_n \equiv K_{n-1}$ (cf.\ Axiom~\ref{ax:norm}).
Property (\ref{eq:coherence}) reflects what is sometimes called
the ``coherence'' of the protocol;
in its absence the protocol becomes a money machine for Skeptic.
(See also Lemma~\ref{lem:referee-1} below.)
Protocol~1 covers the apparently more general case
where the superexpectations $\EEE_n$ are not fixed in advance
but chosen by a third player:
see Section~\ref{sec:generality} below.

\begin{remark}
  In \cite{\Takemura} we considered a different
  but essentially equivalent prediction protocol
  (\cite{\Takemura}, Protocol~2).
  For connections with Peter Walley's theory of imprecise probabilities,
  see the recent article \cite{decooman/hermans:2008}.
\end{remark}

\begin{remark}
  An apparently more general version of Protocol~1
  is where, at each trial $n$,
  World chooses the outcome $x_n$ from a set $\XXX_n$
  which may depend on $n$
  and $\EEE_n$ is an outer probability content on $\XXX_n$.
  However, this version immediately reduces to our Protocol~1
  by setting $\XXX:=\bigcup_{n=1}^{\infty}\XXX_n$
  and extending each $\EEE_n$ to $\bbbrbarXXX$ as
  $\EEE'_n(f):=\EEE_n(f|_{\XXX_n})$,
  $f|_{\XXX_n}$ being the restriction of $f$ to $\XXX_n$.
\end{remark}

We call the set $\Omega:=\XXX^{\infty}$
of all infinite sequences of World's moves the \emph{sample space}.
The elements of the set
$\XXX^*:=\bigcup_{n=0}^{\infty}\XXX^n$
of all finite sequences of World's moves are called \emph{situations}.
For each situation $s$ we let $\Gamma(s)\subseteq\Omega$
stand for the set of all infinite extensions in $\Omega$ of $s$
(i.e., $\Gamma(s)$ is the set of all $\omega\in\Omega$
such that $s$ is a prefix of $\omega$).
Let $\Box$ be the empty situation.
If $s$ is a situation and $x\in\XXX$,
$sx$ is the situation obtained from  $s$ by adding $x$ on the right;
therefore,
$sx=x_1\ldots x_nx$ when $s=x_1\ldots x_n$.
If $s$ and $t$ are two situations,
we write $s\subseteq t$ when $s$ is a prefix of $t$,
and we write $s\subset t$ when $s\subseteq t$ and $s\ne t$.
We will also be using derived notation such as
$s\subseteq u\subseteq t$, $s\not\subseteq t$, and $s\supseteq t$.

The \emph{length} $\lvert s\rvert$ of a situation $s\in\XXX^n$ is $n$
(i.e., $\lvert s\rvert$ is the length of $s$ as a finite sequence);
in particular, $\lvert \Box\rvert = 0$.
If $\omega\in\Omega$ and $n\in\{0,1,\ldots\}$,
$\omega^n$ is defined to be the unique situation of length $n$
that is a prefix of $\omega$.
For $\omega=x_1x_2\ldots\in\Omega$ and $n\in\bbbn$,
we let $\omega_n\in\XXX$ stand for $x_n$.

A strategy $\Sigma$ for Skeptic is a pair $(\Sigma_0,\Sigma_1)$,
where $\Sigma_0\in\bbbrbar$
(informally, this is the initial capital chosen by Skeptic)
and $\Sigma_1:\XXX^*\to\bbbrbarXXX$ is a function
satisfying $\EEE_1(\Sigma_1(\Box))\le\Sigma_0$ and
$\EEE_n(\Sigma_1(sx))\le\Sigma_1(s)(x)$
for all $n\ge2$, $s\in\XXX^{n-2}$, and $x\in\XXX$
(informally, for each situation $s$, $\Sigma_1(s)$ is the function chosen by Skeptic
in situation $s$).
If we fix a strategy $\Sigma=(\Sigma_0,\Sigma_1)$ for Skeptic,
his capital $\K_n$ becomes a function
of the current situation $s=x_1\ldots x_n$ of length $n$.
We write $\K^{\Sigma}(s)$ for $\K_n$
resulting from Skeptic following $\Sigma$ and from World playing $s$.
Formally,
the function $\K^{\Sigma}:\XXX^*\to\bbbrbar$ is defined by
$\K^{\Sigma}(\Box):=\Sigma_0$
and
$
  \K^{\Sigma}(sx)
  :=
  \Sigma_1(s)(x)
$
for all $s\in\XXX^*$ and all $x\in\XXX$.
This function will be called the \emph{capital process} of $\Sigma$.
A function $\SSS$ is called a (game-theoretic) \emph{supermartingale}
if it is the capital process, $\SSS=\K^{\Sigma}$,
of some strategy $\Sigma$ for Skeptic.

Notice that a function $\SSS$ is a supermartingale
if and only if $\SSS:\XXX^*\to\bbbrbar$
and, for all $n\in\bbbn$ and all situations $s\in\XXX^{n-1}$ of length $n-1$,
$$
  \EEE_n
  \left(
    \SSS(s\,\cdot)
  \right)
  \le
  \SSS(s),
$$
where, as usual, $\SSS(s\,\cdot):\XXX\to\bbbrbar$ is the function
mapping each $x\in\XXX$ to $\SSS(sx)$.
A supermartingale $\SSS$ is a \emph{martingale}
if
$
  \EEE_n
  (\SSS(s\,\cdot))
  =
  \SSS(s)
$
for all $n\in\bbbn$ and $s\in\XXX^{n-1}$.

\begin{remark}
  Martingales are less useful for us than supermartingales
  since the sum of two martingales may fail to be a martingale
  (the inequality in Axiom~\ref{ax:sum} may be strict),
  whereas the sum of two supermartingales is always a supermartingale.
  Under Axiom~\ref{ax:countable},
  even a countable sum of positive supermartingales
  is a supermartingale.
\end{remark}

The following useful property of supermartingales
follows easily from (\ref{eq:coherence}).
\begin{lemma}\label{lem:referee-1}
  For each supermartingale $\SSS$ and each situation $s$,
  $$
    \SSS(s)
    \ge
    \inf_{\omega\in\Gamma(s)}
    \overline\SSS(\omega),
  $$
  where $\overline\SSS(\omega)$
  is defined to be
  $\limsup_{n\to\infty}\SSS(\omega^n)$
  for all $\omega\in\Omega$.
\end{lemma}
\begin{proof}
  Let $s=x_1\ldots x_k\in\XXX^k$
  and let $r>\SSS(s)$ be given.
  Since $\EEE_{k+1}(\SSS(s\,\cdot))\le\SSS(s)<r$,
  by (\ref{eq:coherence})
  there exists $x_{k+1}\in\XXX$ such that $\SSS(x_1\ldots x_{k+1})<r$.
  Repeating the argument we can find $x_{k+1}x_{k+2}\ldots$
  such that $\SSS(x_1\ldots x_n)<r$ for all $n\ge k$.
  Setting $\omega:=x_1x_2\ldots$,
  we have $\omega\in\Gamma(s)$ and $\SSS(\omega^n)<r$ for all $n\ge k$.
  This completes the proof.
\end{proof}

For each function $\xi:\Omega\to\bbbrbar$
and each situation $s$,
we define the (conditional) \emph{upper expectation} of $\xi$ given $s$ by
\begin{multline}\label{eq:upper-expectation}
  \UpperExpect(\xi\givn s)
  :=
  \inf
  \Bigl\{
    a
    \bigm|
    \exists\SSS: \SSS(s)=a \text{ and }\\
    \liminf_{n\to\infty} \SSS(\omega^n)
    \ge
    \xi(\omega)
    \text{ for all $\omega\in\Gamma(s)$}
  \Bigr\}
\end{multline}
where $\SSS$ ranges over the supermartingales that are bounded below,
and we define the \emph{lower expectation} of $\xi$ given $s$ by
\begin{equation}\label{eq:lower-expectation}
  \LowerExpect(\xi\givn s)
  :=
  -
  \UpperExpect
  \left(
    -\xi \givn s
  \right).
\end{equation}
If $E$ is any subset of $\Omega$,
its \emph{upper} and \emph{lower probability}
given a situation $s$ are defined by
\begin{equation}\label{eq:probability}
  \UpperProb(E\givn s):=\UpperExpect(\III_E\givn s),
  \quad
  \LowerProb(E\givn s):=\LowerExpect(\III_E\givn s),
\end{equation}
respectively.
In what follows we sometimes refer
to sets $E\subseteq\Omega$ as \emph{events}.
\begin{lemma}\label{lem:referee-3}
  For each situation $s$,
  $\UpperExpect(\cdot\givn s):\bbbrbarOmega\to\bbbrbar$
  is an outer probability content.
\end{lemma}
\begin{proof}
  It is evident that Axiom~\ref{ax:order} is satisfied
  for $\UpperExpect(\cdot\givn s)$.
  Axiom~\ref{ax:sum} for $\UpperExpect(\cdot\givn s)$
  follows from the fact
  that, by Axiom~\ref{ax:sum} applied to $\EEE_1,\EEE_2,\ldots$,
  the sum $\UUU:=\SSS+\TTT$
  of two bounded below supermartingales $\SSS$ and $\TTT$
  is again a bounded below supermartingale
  and that it satisfies
  $$
    \liminf_{n\to\infty}\SSS(\omega^n)
    +
    \liminf_{n\to\infty}\TTT(\omega^n)
    \le
    \liminf_{n\to\infty}\UUU(\omega^n)
  $$
  for all $\omega\in\Gamma(s)$.
  In the same manner,
  we can use Axiom~\ref{ax:scaling} applied to $\EEE_1,\EEE_2,\ldots$
  to deduce that $\UpperExpect(\cdot\givn s)$ satisfies Axiom~\ref{ax:scaling}.
  Let $c\in\bbbr$.
  Since the function on $\bbbrbarXXX$ that is identically equal to $c$
  is a supermartingale,
  $\UpperExpect(c\givn s)\le c$.
  Let $\SSS$ be a bounded below supermartingale satisfying
  $\liminf_{n\to\infty}\SSS(\omega^n)\ge c$ for all $\omega\in\Gamma(s)$.
  By Lemma~\ref{lem:referee-1},
  we then have $\SSS(s)\ge c$.
  Therefore, $\UpperExpect(c\givn s)\ge c$,
  which completes the proof of Axiom~\ref{ax:norm}
  for $\UpperExpect(\cdot\givn s)$.
\end{proof}
Lemma~\ref{lem:referee-3} immediately implies the following statement
(cf.\ \cite{hoffmann-jorgensen:1987}, (5.4)).
\begin{corollary}\label{cor:coherence}
  For all situations $s$ and all functions $\xi:\Omega\to\bbbrbar$,
  $\LowerExpect(\xi\givn s)\le\UpperExpect(\xi\givn s)$.
  In particular,
  $\LowerProb(E\givn s)\le\UpperProb(E\givn s)$
  for all events $E\subseteq\Omega$.
\end{corollary}
\begin{proof}
  Suppose $\LowerExpect(\xi\givn s)>\UpperExpect(\xi\givn s)$,
  i.e.,
  $\UpperExpect(\xi\givn s)+\UpperExpect(-\xi\givn s)<0$.
  By Axiom~\ref{ax:sum} applied to $\UpperExpect(\cdot\givn s)$
  this implies $\UpperExpect(0\givn s)<0$,
  which contradicts Axiom~\ref{ax:norm} for $\UpperExpect(\cdot\givn s)$.
\end{proof}

Important special cases are where $s=\Box$
(unconditional upper and lower expectations and probabilities).
We set
$\UpperExpect(\xi):=\UpperExpect(\xi\givn\Box)$,
$\LowerExpect(\xi):=\LowerExpect(\xi\givn\Box)$,
$\UpperProb(E):=\UpperProb(E\givn\Box)$,
and $\LowerProb(E):=\LowerProb(E\givn\Box)$.
We say that an event $E$ is \emph{almost certain},
or happens \emph{almost surely} (a.s.),
if $\LowerProb(E)=1$;
in this case we will also say that $E$,
considered as a property of $\omega\in\Omega$,
holds for \emph{almost all} $\omega$.
More generally, we say that $E$ holds \emph{almost surely on $B$}
(or \emph{for almost all $\omega\in B$}),
for another event $B$,
if the event $(B\Rightarrow E):=(B^c\cup E)$ is almost certain.
An event $E$ is \emph{almost impossible}, or \emph{null},
if $\UpperProb(E)=0$.

In \cite{shafer/vovk:2001} we defined the lower probability of an event $E$
as $1-\UpperProb(E^c)$.
The following lemma says that this definition
is equivalent to our current definition.
\begin{lemma}
  For each event $E\subseteq\Omega$ and each situation $s$,
  \begin{equation*}
    \LowerProb(E\givn s)
    =
    1-\UpperProb(E^c\givn s).
  \end{equation*}
\end{lemma}
\begin{proof}
  By~(\ref{eq:constant}) and Lemma~\ref{lem:referee-3},
  we have
  $\UpperExpect(\xi+c\givn s)=\UpperExpect(\xi\givn s)+c$
  for all $\xi:\Omega\to\bbbr$ and $c\in\bbbr$.
  Therefore,
  \begin{multline*}
    \LowerProb(E\givn s)
    =
    \LowerExpect(\III_E\givn s)
    =
    -\UpperExpect(-\III_E\givn s)\\
    =
    1-\UpperExpect(1-\III_E\givn s)
    =
    1-\UpperExpect(\III_{E^c}\givn s)
    =
    1-\UpperProb(E^c\givn s).
    \qedhere
  \end{multline*}
\end{proof}

The following lemma will be used
in the proof of Lemma~\ref{lem:referee-4}
stating that upper expectation is a superexpectation functional
in the case where $\EEE_1,\EEE_2,\ldots$ are superexpectation functionals.
\begin{lemma}\label{lem:referee-2}
  The right-hand side of (\ref{eq:upper-expectation})
  will not change when $\Gamma(s)$ is replaced by~$\Omega$.
\end{lemma}
\begin{proof}
  It suffices to prove that if a bounded below supermartingale $\SSS$
  satisfies $\SSS(s)<r$
  and $\liminf_{n\to\infty}\SSS(\omega^n)\ge\xi(\omega)$ for all $\omega\in\Gamma(s)$,
  then there exists another bounded below supermartingale $\SSS'$
  that satisfies $\SSS'(s)<r$
  and $\liminf_{n\to\infty}\SSS'(\omega^n)\ge\xi(\omega)$ for all $\omega\in\Omega$.
  Such an $\SSS'$ can be defined by
  $$
    \SSS'(t)
    :=
    \begin{cases}
      \SSS(t) & \text{if $s\subseteq t$}\\
      \infty & \text{otherwise}.
    \end{cases}
    \qedhere
  $$
\end{proof}
\begin{lemma}\label{lem:referee-4}
  Suppose $\EEE_1,\EEE_2,\ldots$ are superexpectations,
  and let $s$ be a situation.
  Then $\UpperExpect(\cdot\givn s)$ is also a superexpectation.
  In particular, for any sequence of events $E_1,E_2,\ldots$,
  it is true that
  \begin{equation*}
    \UpperProb
    \left(
      \bigcup_{k=1}^{\infty}
      E_k
    \right)
    \le
    \sum_{k=1}^{\infty}
    \UpperProb
    \left(
      E_k
    \right).
  \end{equation*}
  In particular,
  the union of a sequence of null events is null.
\end{lemma}
\begin{proof}
  In view of Lemma~\ref{lem:referee-3},
  our goal is to prove
  \begin{equation*} 
    \UpperExpect
    \left(
      \sum_{k=1}^{\infty}
      \xi_k
      \Bigm|
      s
    \right)
    \le
    \sum_{k=1}^{\infty}
    \UpperExpect
    \left(
      \xi_k \givn s
    \right),
  \end{equation*}
  where $\xi_1,\xi_2,\ldots$ are positive functions in $\bbbrbarOmega$.
  Let $\epsilon>0$ be arbitrarily small.
  For each $k\in\bbbn$ choose a supermartingale $\SSS_k$
  (automatically positive, by Lemma~\ref{lem:referee-1})
  such that
  $\liminf_n\SSS_k(\omega^n)\ge\xi_k(\omega)$
  for all $\omega\in\Omega$
  (cf.\ Lemma~\ref{lem:referee-2})
  and $\SSS_k(s)\le\UpperExpect(\xi_k\givn s)+\epsilon/2^k$.
  Since all $\EEE_n$ are superexpectations,
  the sum $\SSS:=\sum_{k=1}^{\infty}\SSS_k$
  will be a supermartingale
  (cf.\ (\ref{eq:subadditivity}));
  this supermartingale will satisfy
  $\SSS(s)\le\sum_{k=1}^{\infty}\UpperExpect(\xi_k\givn s)+\epsilon$
  and, by Fatou's lemma,
  $$
    \liminf_n\SSS(\omega^n)
    =
    \liminf_n
    \sum_{k=1}^{\infty}\SSS_k(\omega^n)
    \ge
    \sum_{k=1}^{\infty}
    \liminf_n\SSS_k(\omega^n)
    \ge
    \sum_{k=1}^{\infty}
    \xi_k(\omega)
  $$
  for all $\omega\in\Omega$.
  Therefore,
  \begin{equation*}
    \UpperExpect
    \left(
      \sum_{k=1}^{\infty}
      \xi_k
      \Bigm|
      s
    \right)
    \le
    \SSS(s)
    \le
    \sum_{k=1}^{\infty}
    \UpperExpect
    \left(
      \xi_k \givn s
    \right)
    +
    \epsilon.
  \end{equation*}
  It remains to remember that $\epsilon$ can be taken arbitrarily small.
\end{proof}

\section{Equivalent definitions of game-theoretic expectation and probability}
\label{sec:equivalent}

The following proposition,
which is our main statement of equivalence,
gives two equivalent definitions of upper game-theoretic expectation.
\begin{theorem}\label{thm:expectation}
  For all $\xi:\Omega\to\bbbrbar$ and all situations $s$,
  \begin{equation*} 
    \UpperExpect(\xi\givn s)
    =
    \inf
    \Bigl\{
      \SSS(s)
      \bigm|
      \limsup_{n\to\infty} \SSS(\omega^n)
      \ge
      \xi(\omega)
      \text{ for all $\omega\in\Gamma(s)$}
    \Bigr\}
  \end{equation*}
  (i.e., on the right-hand side of (\ref{eq:upper-expectation}),
  we can replace $\liminf$ by $\limsup$),
  where $\SSS$ ranges over the supermartingales that are bounded below,
  and
  \begin{equation*} 
    \UpperExpect(\xi\givn s)
    =
    \inf
    \left\{
      \SSS(s)
      \bigm|
      \forall\omega\in\Gamma(s):
      \lim_{n\to\infty}
      \SSS(\omega^n)
      \ge
      \xi(\omega)
    \right\}
  \end{equation*}
  where $\SSS$ ranges over the class $\mathbf{L}$
  of all bounded below supermartingales for which
  $\lim_{n\to\infty}\SSS(\omega^n)$
  exists in $(-\infty,\infty]$ for all $\omega\in\Omega$.
\end{theorem}
\begin{proof}
  Let a bounded below supermartingale $\SSS$ satisfy the inequality
  \begin{equation*}
    \forall\omega\in\Gamma(s_0):
    \limsup_{n\to\infty}
    \SSS(\omega^n)
    \ge
    \xi(\omega)
  \end{equation*}
  (cf.\ (\ref{eq:upper-expectation})) and let $\epsilon\in(0,1)$.
  It suffices to show that there exists $\SSS^*\in\mathbf{L}$ such that
  \begin{equation}\label{eq:star}
    \SSS^*(s_0)\le\SSS(s_0)+\epsilon
    \text{\quad and\quad}
    \forall\omega\in\Gamma(s_0):
    \lim_{n\to\infty}
    \SSS^*(\omega^n)
    \ge
    \xi(\omega).
  \end{equation}
  Without loss of generality
  we assume $\SSS(s_0)<\infty$
  (if $\SSS(s_0)=\infty$, set $\SSS^*(s):=\infty$ for all $s$).
  Setting $\SSS':=(\SSS-C)/(\SSS(s_0)-C)$,
  where $C$ is any constant satisfying $C<\inf\SSS$,
  we obtain a positive supermartingale satisfying $\SSS'(s_0)=1$.
  
  The idea is now to use the standard proof of Doob's convergence theorem
  (see, e.g., \cite{shafer/vovk:2001}, Lemma 4.5).
  But first we need to give more definitions,
  which will also be used in the proof of Theorem~\ref{thm:levy}.

  If $s$ and $t$ are two situations such that $s\subseteq t$,
  we define the ``intervals''
  \begin{align*}
    [s,t] &:= \{u\st s\subseteq u\subseteq t\},&
    [s,t) &:= \{u\st s\subseteq u\subset t\},\\
    (s,t] &:= \{u\st s\subset u\subseteq t\},&
    (s,t) &:= \{u\st s\subset u\subset t\}.
  \end{align*}
  Two situations $s$ and $t$ are said to be \emph{comparable}
  if $s\subseteq t$ or $t\subseteq s$;
  otherwise they are \emph{incomparable}.
  A \emph{cut} is a set of situations that are pairwise incomparable
  (cuts are analogous to stopping times in measure-theoretic probability).
  If $\sigma$ and $\tau$ are two cuts,
  we write $\sigma\le\tau$ to mean
  $
    \forall t\in\tau \,
    \exists s\in\sigma:
    s\subseteq t
  $,
  and we write $\sigma<\tau$ to mean
  $
    \forall t\in\tau \,
    \exists s\in\sigma:
    s\subset t
  $.
  In the case $\sigma\le\tau$, we define the ``time intervals''
  \begin{align*}
    [\sigma,\tau] &:= \{u\st [\Box,u]\cap\sigma\ne\emptyset, [\Box,u)\cap\tau=\emptyset\},\\
    [\sigma,\tau) &:= \{u\st [\Box,u]\cap\sigma\ne\emptyset, [\Box,u]\cap\tau=\emptyset\},\\
    (\sigma,\tau] &:= \{u\st [\Box,u)\cap\sigma\ne\emptyset, [\Box,u)\cap\tau=\emptyset\},\\
    (\sigma,\tau) &:= \{u\st [\Box,u)\cap\sigma\ne\emptyset, [\Box,u]\cap\tau=\emptyset\}.
  \end{align*}
  Notice that for all stopping times $\sigma,\tau,\rho$
  such that $\sigma \le \tau \le \rho$,
  $$
    [\sigma,\tau) \cap [\tau,\rho) = \emptyset
    \text{ and }
    [\sigma,\tau) \cup [\tau,\rho)
    =
    [\sigma,\rho).
  $$

  If $s$ is a situation and $\tau$ is a cut,
  $s^{\tau}$ stands for the unique (when it exists) situation $t\in\tau$
  such that $t\subseteq s$.
  Similarly, if $\omega\in\Omega$ and $\tau$ is a cut,
  $\omega^{\tau}$ stands for the unique (when it exists) situation $t\in\tau$
  that is a prefix of $\omega$.
  (The case where $\omega^{\tau}$ does not exist
  is analogous to the case where a stopping time takes value $\infty$
  in measure-theoretic probability.)
  Notice that our notation $\omega^n$ for $n=0,1,\ldots$
  can be regarded as a special case of the new notation:
  we can interpret the upper index $n$
  as the cut consisting of all situations of length $n$.
  We will also be using the notation $s^n$,
  where $n=0,1,\ldots$ and $s$ is a situation,
  in the same sense.

  Let $[a_i,b_i]$, $i=1,2,\ldots$,
  be an enumeration of all intervals with $0\le a_i<b_i<\infty$ and both end-points rational.
  For each $i$ one can define a positive supermartingale $\SSS^i$ with
  $
    \SSS^i(s_0) = 1
  $
  such that $\SSS^i(\omega^n)$ converges to $\infty$ as $n\to\infty$
  when $\liminf_n\SSS'(\omega^n)<a_i$ and $\limsup_n\SSS'(\omega^n)>b_i$.
  The construction of $\SSS^i$ is standard.
  First we define two sequences of cuts
  $\tau^i_0,\tau^i_1,\ldots$ and $\sigma^i_1,\sigma^i_2,\ldots$
  by setting $\tau^i_0:=\{s_0\}$ and, for $k=1,2,\ldots$,
  \begin{align*} 
    \sigma^i_k
    &:=
    \{s \st \SSS'(s)>b_i, \exists t\subset s: t\in\tau^i_{k-1},
      \forall u\in(t,s): \SSS'(u)\le b_i\},\\
    \tau^i_k
    &:=
    \{s \st \SSS'(s)<a_i, \exists t\subset s: t\in\sigma^i_k,
      \forall u\in(t,s): \SSS'(u)\ge a_i\}.
  \end{align*}
  Now we define $\SSS^i$ by the requirement that,
  for all situations $s\supseteq s_0$ and all $x\in\XXX$,
  \begin{equation}\label{eq:P}
    \SSS^i(sx)
    :=
    \begin{cases}
      \SSS^i(s)+\SSS'(sx)-\SSS'(s)
      & \text{if $\SSS^i(s)<\infty$
	and $\exists k: s\in[\tau^i_{k-1},\sigma^i_k)$}\\
      \SSS^i(s)
      & \text{otherwise};
    \end{cases}
  \end{equation}
  in conjunction with $\SSS^i(s_0) = 1$ this determines $\SSS^i$ uniquely
  on the situations $s\supseteq s_0$.
  If $s\not\supseteq s_0$, set $\SSS^i(s):=\infty$.
  We have $\EEE_{n}(\SSS^i(s\,\cdot))\le\SSS^i(s)$,
  where $n:=\lvert s\rvert+1$, in both cases considered in (\ref{eq:P});
  e.g., since $\SSS'$ is a supermartingale,
  $$
    \EEE_{n}(\SSS^i(s\,\cdot))
    =
    \EEE_{n}(\SSS^i(s)+\SSS'(s\,\cdot)-\SSS'(s))
    \le
    \SSS^i(s)
  $$
  when $\SSS^i(s)<\infty$ and $\exists k: s\in[\tau^i_{k-1},\sigma^i_k)$.

  Let us check that each supermartingale $\SSS^i$ is positive.
  There are three (overlapping) cases:
  \begin{itemize}
  \item
    If $s\in[\tau^i_0,\sigma^i_1]$,
    $$
      \SSS^i(s) \ge \SSS'(s) \ge 0
    $$
    (we write $\SSS^i(s) \ge \SSS'(s)$ rather than $\SSS^i(s) = \SSS'(s)$
    because of the possibility that $\SSS'(s)<\infty$
    but $\SSS'(t)=\infty$ for some $t\subset s$).
  \item
    If $s\in[\sigma^i_k,\tau^i_k]$ for some $k=1,2,\ldots$,
    \begin{multline}\label{eq:frozen}
      \SSS^i(s)
      \ge
      1
      +
      \left(
        \SSS'(s^{\sigma^i_1}) - \SSS'(s_0)
      \right)
      +
      \left(
        \SSS'(s^{\sigma^i_2}) - \SSS'(s^{\tau^i_1})
      \right)\\
      +\cdots+
      \left(
        \SSS'(s^{\sigma^i_k}) - \SSS'(s^{\tau^i_{k-1}})
      \right)
      \ge
      b_i + (k-1)(b_i-a_i)
      \ge
      0.
    \end{multline}
  \item
    If $s\in[\tau^i_k,\sigma^i_{k+1}]$ for some $k=1,2,\ldots$,
    \begin{multline}\label{eq:moving}
      \SSS^i(s)
      \ge
      1
      +
      \left(
        \SSS'(s^{\sigma^i_1}) - \SSS'(s_0)
      \right)
      +
      \left(
        \SSS'(s^{\sigma^i_2}) - \SSS'(s^{\tau^i_1})
      \right)\\
      +\cdots+
      \left(
        \SSS'(s^{\sigma^i_k}) - \SSS'(s^{\tau^i_{k-1}})
      \right)
      +
      \left(
        \SSS'(s) - \SSS'(s^{\tau^i_k})
      \right)\\
      \ge
      b_i + (k-1)(b_i-a_i) + \SSS'(s) - a_i
      \ge
      k(b_i-a_i)
      \ge
      0.
    \end{multline}
  \end{itemize}
  Equations~(\ref{eq:frozen}) and~(\ref{eq:moving}) also show
  that $\SSS^i(\omega^n)$ indeed converges to $\infty$ as $n\to\infty$
  whenever $\liminf_n\SSS'(\omega^n)<a_i$ and $\limsup_n\SSS'(\omega^n)>b_i$
  for $\omega\in\Gamma(s_0)$.

  Now we can set
  \begin{equation}\label{eq:T}
    \TTT
    :=
    \sum_{i=1}^{\infty}
    2^{-i}
    \SSS^i
  \end{equation}
  and $\SSS^*:=\SSS+\epsilon\TTT$.
  Assume, for a moment, that $\EEE_1,\EEE_2,\ldots$ are superexpectations.
  In this case $\TTT$, being a countable sum of positive supermartingales,
  is a positive supermartingale itself.
  Being a sum of two supermartingales,
  $\SSS^*$ is a supermartingale itself.

  Let us check that $\SSS^*\in\mathbf{L}$;
  this will imply the second inequality in~(\ref{eq:star})
  (the first inequality holds by the definition of $\SSS^*$).
  Since $\SSS^*$ is bounded below,
  we are only required to check that $\SSS^*(\omega^n)$ converges in $(-\infty,\infty]$
  as $n\to\infty$ for all $\omega\in\Omega$.
  Fix $\omega\in\Omega$.

  If $\SSS(\omega^n)=\infty$ for some $n$,
  there exists $i$ such that $\SSS^i(\omega^n)=\infty$ from some $n$ on
  (take any $i$ such that $a_i=0$ and $b_i>\max_{k<n}\SSS'(\omega^k)$,
  where $n$ is the smallest number such that $\SSS(\omega^n)=\infty$),
  and so we have $\TTT(\omega^n)=\infty$
  and $\SSS^*(\omega^n)=\infty$ from some $n$ on.
  Therefore, we will assume that $\SSS(\omega^n)<\infty$ for all $n$.

  If $\SSS(\omega^n)$ converges to $\infty$,
  $\SSS^*(\omega^n)$ also converges to $\infty$.
  If $\SSS(\omega^n)$ (and, therefore, $\SSS'(\omega^n)$)
  does not converge in $(-\infty,\infty]$,
  there exists $i$ such that $\SSS^i(\omega^n)\to\infty$
  (take any $i$ satisfying
  $\liminf_n\SSS'(\omega^n)<a_i<b_i<\limsup_n\SSS'(\omega^n)$),
  and so we have $\TTT(\omega^n)\to\infty$
  and $\SSS^*(\omega^n)\to\infty$.
  It remains to consider the case where $\SSS(\omega^n)$
  converges in $\bbbr$.

  Suppose $\SSS(\omega^n)$ and, therefore, $\SSS'(\omega^n)$
  converge in $\bbbr$
  but $\SSS^*(\omega^n)$ does not converge in $(-\infty,\infty]$.
  Choose a non-empty interval $(a,b)\subseteq\bbbr$
  such that
  $\liminf_n\SSS^*(\omega^n)<a<b<\limsup_n\SSS^*(\omega^n)$
  and set $c:=b-a$.
  Take any $N\in\bbbn$ such that $\SSS^*(\omega^N)>b$
  and $\lvert\SSS(\omega^n)-\SSS(\omega^m)\rvert<c/2$,
  $\lvert\SSS'(\omega^n)-\SSS'(\omega^m)\rvert<c/4$ for all $n,m\ge N$.
  Since $\SSS'(\omega^n)-\SSS'(\omega^m)>-c/4$ for all $n,m\ge N$,
  we will have
  \begin{equation}\label{eq:pedantic}
    \SSS^i(\omega^n)-\SSS^i(\omega^N)>-c/2
  \end{equation}
  for all $i$ and all $n\ge N$.
  Indeed, there are five cases (overlapping):
  \begin{itemize}
  \item 
    If $\omega^N\in[\sigma^i_l,\tau^i_l]$ for some $l=1,2,\ldots$
    and $\omega^n\in[\sigma^i_k,\tau^i_k]$ for some $k=l,l+1,\ldots$:
    \begin{multline*}
      \SSS^i(\omega^n) - \SSS^i(\omega^N)
      =
      \left(
        \SSS'(\omega^{\sigma^i_{l+1}}) - \SSS'(\omega^{\tau^i_l})
      \right)
      +
      \left(
        \SSS'(\omega^{\sigma^i_{l+2}}) - \SSS'(\omega^{\tau^i_{l+1}})
      \right)\\
      +\cdots+
      \left(
        \SSS'(\omega^{\sigma^i_k}) - \SSS'(\omega^{\tau^i_{k-1}})
      \right)
      \ge
      (k-l)(b_i-a_i)
      \ge
      0.
    \end{multline*}
  \item 
    If $\omega^N\in[\sigma^i_l,\tau^i_l]$ for some $l=1,2,\ldots$
    and $\omega^n\in[\tau^i_k,\sigma^i_{k+1}]$ for some $k=l,l+1,\ldots$:
    \begin{multline*}
      \SSS^i(\omega^n) - \SSS^i(\omega^N)
      =
      \left(
        \SSS'(\omega^{\sigma^i_{l+1}}) - \SSS'(\omega^{\tau^i_l})
      \right)
      +
      \left(
        \SSS'(\omega^{\sigma^i_{l+2}}) - \SSS'(\omega^{\tau^i_{l+1}})
      \right)\\
      +\cdots+
      \left(
        \SSS'(\omega^{\sigma^i_k}) - \SSS'(\omega^{\tau^i_{k-1}})
      \right)
      +
      \left(
        \SSS'(\omega^n) - \SSS'(\omega^{\tau^i_k})
      \right)\\
      >
      (k-l)(b_i-a_i) - c/4
      \ge
      -c/4.
    \end{multline*}
  \item 
    If $\omega^N\in[\tau^i_{l-1},\sigma^i_l]$ for some $l=1,2,\ldots$
    and $\omega^n\in[\sigma^i_k,\tau^i_k]$ for some $k=l,l+1,\ldots$:
    \begin{multline*}
      \SSS^i(\omega^n) - \SSS^i(\omega^N)
      =
      \left(
        \SSS'(\omega^{\sigma^i_l}) - \SSS'(\omega^N)
      \right)
      +
      \left(
        \SSS'(\omega^{\sigma^i_{l+1}}) - \SSS'(\omega^{\tau^i_l})
      \right)
      +\\
      \left(
        \SSS'(\omega^{\sigma^i_{l+2}}) - \SSS'(\omega^{\tau^i_{l+1}})
      \right)
      +\cdots+
      \left(
        \SSS'(\omega^{\sigma^i_k}) - \SSS'(\omega^{\tau^i_{k-1}})
      \right)\\
      >
      -c/4 + (k-l)(b_i-a_i)
      \ge
      -c/4.
    \end{multline*}
  \item 
    If $\omega^N\in[\tau^i_{l-1},\sigma^i_l]$ for some $l=1,2,\ldots$
    and $\omega^n\in[\tau^i_k,\sigma^i_{k+1}]$ for some $k=l,l+1,\ldots$:
    \begin{multline*}
      \SSS^i(\omega^n) - \SSS^i(\omega^N)
      =
      \left(
        \SSS'(\omega^{\sigma^i_l}) - \SSS'(\omega^N)
      \right)
      +
      \left(
        \SSS'(\omega^{\sigma^i_{l+1}}) - \SSS'(\omega^{\tau^i_l})
      \right)\\
      +
      \left(
        \SSS'(\omega^{\sigma^i_{l+2}}) - \SSS'(\omega^{\tau^i_{l+1}})
      \right)
      +\cdots+
      \left(
        \SSS'(\omega^{\sigma^i_k}) - \SSS'(\omega^{\tau^i_{k-1}})
      \right)\\
      +
      \left(
        \SSS'(\omega^n) - \SSS'(\omega^{\tau^i_k})
      \right)
      >
      -c/4 + (k-l)(b_i-a_i) - c/4
      \ge
      -c/2.
    \end{multline*}
  \item 
    If $\omega^N,\omega^n\in[\tau^i_{l-1},\sigma^i_l]$ for some $l=1,2,\ldots$:
    \begin{equation*}
      \SSS^i(\omega^n) - \SSS^i(\omega^N)
      =
      \SSS'(\omega^n) - \SSS'(\omega^N)
      >
      -c/4.
    \end{equation*}
  \end{itemize}
  In all five cases we have (\ref{eq:pedantic}).
  This implies
  $\TTT(\omega^n)-\TTT(\omega^N)>-c/2$,
  and so
  $\SSS^*(\omega^n)-\SSS^*(\omega^N)>-c$
  for all $n\ge N$
  (remember that $\epsilon<1$).
  The latter contradicts the fact that
  $\SSS^*(\omega^n)-\SSS^*(\omega^N)<-c$
  for some $n\ge N$
  (namely, for any $n\ge N$ satisfying
  $\SSS^*(\omega^n)<a$).

  This completes the proof in the case
  where $\EEE_1,\EEE_2,\ldots$ are superexpectations.
  However, we do not really need Axiom \ref{ax:countable}:
  despite the appearance of an infinite sum in (\ref{eq:T}),
  for each situation $s$ and each $x\in\XXX$
  the increment $\TTT(sx)-\TTT(s)$ of $\TTT$
  can be represented (assuming $\TTT(s)<\infty$) as
  \begin{equation*}
    \TTT(sx)-\TTT(s)
    =
    \sum_{i=1}^{\infty}
    2^{-i}
    (\SSS^i(sx)-\SSS^i(s))
    =
    \left(
      \sum_{i=1}^{\infty}
      w_i
    \right)
    (\SSS'(sx)-\SSS'(s)),
  \end{equation*}
  where $w_i\in\{0,2^{-i}\}$
  (this makes the series $\sum_{i=1}^{\infty}w_i$ convergent in $\bbbr$)
  are defined by
  $$
    w_i
    :=
    \begin{cases}
      2^{-i} & \text{if $\exists k: s\in[\tau^i_{k-1},\sigma^i_k)$}\\
      0 & \text{otherwise}.
    \end{cases}
  $$
  Since $\SSS'$ is a supermartingale,
  $\EEE_{\lvert s\rvert+1}(\TTT(s\,\cdot)-\TTT(s))\le0$.
  This argument for $\TTT$ being a supermartingale
  does not depend on Axiom \ref{ax:countable}.
\end{proof}

Replacing the $\liminf_{n\to\infty}$ in (\ref{eq:upper-expectation})
by $\inf_n$ or $\sup_n$ does change the definition.
If we replace the $\liminf_{n\to\infty}$ by $\inf_n$,
we will have $\UpperExpect(\xi\givn s)=\sup_{\omega\in\Gamma(s)}\xi(\omega)$.
In the following example
we consider replacing $\liminf_{n\to\infty}$ by $\sup_n$.
\begin{example}\label{ex:expectation}
  Set
  \begin{equation*}
    \UpperExpect_1(\xi)
    :=
    \inf
    \left\{
      \SSS(\Box)
      \Bigm|
      \forall\omega\in\Omega:
      \sup_n
      \SSS(\omega^n)
      \ge
      \xi(\omega)
    \right\},
  \end{equation*}
  $\SSS$ ranging over the bounded below supermartingales.
  It is always true that $\UpperExpect_1(\xi)\le\UpperExpect(\xi)$.
  Consider the coin-tossing protocol (\cite{shafer/vovk:2001}, Section 8.2),
  which is the special case of Protocol~1 with $\XXX=\{0,1\}$
  and $\EEE_n(f)=(f(0)+f(1))/2$ for all $n\in\bbbn$ and $f\in\bbbrbarXXX$.
  For each $\epsilon\in(0,1]$ there exists a bounded positive function
  $\xi$ on $\Omega$
  such that $\UpperExpect(\xi)=1$ and $\UpperExpect_1(\xi)=\epsilon$.
\end{example}
\begin{proof}
  Let us demonstrate the following equivalent statement:
  for any $C\ge1$ there exists a bounded positive function $\xi$
  such that $\UpperExpect_1(\xi)=1$ and $\UpperExpect(\xi)=C$.
  Fix such a $C$.
  Define $\Xi:\Omega\to[0,\infty]$
  by the requirement $\Xi(\omega):=2^n$
  where $n$ is the number of $1$s at the beginning of $\omega$:
  $
    n:=\max\{i\st\omega_1=\cdots=\omega_i=1\}
  $.
  It is obvious that $\UpperExpect_1(\Xi)=1$ and $\UpperExpect(\Xi)=\infty$.
  However, $\Xi$ is unbounded.
  We can always find $A\ge1$ such that $\UpperExpect(\min(\Xi,A))=C$
  (as the function $a\mapsto\UpperExpect(\min(\Xi,a))$ is continuous).
  Since $\UpperExpect_1(\min(\Xi,A))=1$,
  we can set $\xi:=\min(\Xi,A)$.
\end{proof}

Game-theoretic probability is a special case of game-theoretic expectation,
and in this special case
it is possible to replace $\liminf_{n\to\infty}$ not only by $\limsup_{n\to\infty}$
but also by $\sup_n$,
provided we restrict our attention to positive supermartingales
(simple examples show that this qualification is necessary).
By Lemma~\ref{lem:referee-2},
the definition of conditional upper probability $\UpperProb$
can be rewritten as
\begin{equation}\label{eq:upper-probability}
  \UpperProb(E\givn s)
  :=
  \inf
  \left\{
    \SSS(s)
    \bigm|
    \liminf_{n\to\infty} \SSS(\omega^n)
    \ge
    1
    \text{ for all $\omega\in E\cap\Gamma(s)$}
  \right\},
\end{equation}
$\SSS$ ranging over the positive supermartingales.
\begin{lemma}\label{lem:probability}
  The definition of upper probability
  will not change if we replace the $\liminf_{n\to\infty}$
  in (\ref{eq:upper-probability})
  by $\limsup_{n\to\infty}$ or by $\sup_n$.
\end{lemma}
\noindent
It is obvious that the definition will change
if we replace the $\liminf_{n\to\infty}$ by $\inf_n$:
in this case we will have
$$
  \UpperProb(E\givn s)
  =
  \begin{cases}
    0 & \text{if $E\cap\Gamma(s)=\emptyset$}\\
    1 & \text{otherwise}.
  \end{cases}
$$
\begin{proof}[Proof of Lemma~\ref{lem:probability}]
  It suffices to prove that the definition will not change
  if we replace the $\liminf_{n\to\infty}$ in (\ref{eq:upper-probability})
  by $\sup_n$.
  Consider a strategy for Skeptic resulting in a positive capital process.
  If this strategy ensures
  $
    \sup_n
    \K_n
    >
    1
  $
  when $x_1x_2\ldots\in E\cap\Gamma(s)$
  (it is obvious that it does not matter whether we have $\ge$ or $>$
  in (\ref{eq:upper-probability})),
  Skeptic can also ensure
  $
    \liminf_{n\to\infty}
    \K_n
    >
    1
  $
  when $x_1x_2\ldots\in E\cap\Gamma(s)$
  by stopping
  (i.e., always choosing $f_n$ identically equal to his current capital)
  after his capital $\K_n$ exceeds~$1$.
\end{proof}

\begin{remark}\label{rem:no-loss}
  The basic notion of this article is that of a bounded below supermartingale;
  in particular,
  upper and lower expectation and probability
  are defined in terms of bounded below supermartingales.
  To define the latter it would be sufficient
  to start, instead of outer probability contents,
  from functionals $\FFF$ that satisfy Axioms~\ref{ax:order}--\ref{ax:norm}
  and whose domain consists of the functions in $\bbbrbarX$
  that are bounded from below.
  To see that no generality is lost when one starts from outer probability contents,
  it is sufficient to check that any such $\FFF$ can be extended
  to an outer probability content.
  One possible extension is $\EEE(f):=\lim_{a\to-\infty}\FFF(\max(f,a))$.
  Axioms~\ref{ax:order}--\ref{ax:norm} are easy to check for $\EEE$;
  e.g., Axiom~\ref{ax:sum} follows from the inequality
  $\max(f+g,2a)\le\max(f,a)+\max(g,a)$ and Axiom~\ref{ax:sum} for $\FFF$:
  \begin{multline*}
    \EEE(f+g)
    =
    \lim_{a\to-\infty}\FFF(\max(f+g,a))
    =
    \lim_{a\to-\infty}\FFF(\max(f+g,2a))\\
    \le
    \lim_{a\to-\infty}\FFF(\max(f,a)+\max(g,a))
    \le
    \lim_{a\to-\infty}
    \left(
      \FFF(\max(f,a))
      +
      \FFF(\max(g,a))
    \right)\\
    =
    \lim_{a\to-\infty}
    \FFF(\max(f,a))
    +
    \lim_{a\to-\infty}
    \FFF(\max(g,a))
    =
    \EEE(f)+\EEE(g).
  \end{multline*}
  It is also easy to check that $\EEE$ will be a superexpectation functional
  whenever $\FFF$ is a superexpectation functional.
\end{remark}

\section{L\'evy's zero-one law}
\label{sec:levy}

The following simple theorem is our main result.
\begin{theorem}\label{thm:levy}
  Let $\xi:\Omega\to(-\infty,\infty]$ be bounded from below.
  For almost all $\omega\in\Omega$,
  \begin{equation}\label{eq:goal}
    \liminf_{n\to\infty}
    \UpperExpect(\xi\givn\omega^n)
    \ge
    \xi(\omega).
  \end{equation}
\end{theorem}
\noindent
This theorem is a game-theoretic version of L\'evy's zero-one law.
Its name derives from its well-known connections with various zero-one phenomena,
some of which will be explored in the next section\ifnotJOURNAL\
  and Section \ref{sec:bartfai-revesz}\fi.

\begin{proof}[Proof of Theorem \ref{thm:levy}]
  This proof will be similar to the proof of Theorem~\ref{thm:expectation}
  in that it will be based on the idea used in the standard proof
  of Doob's martingale convergence theorem.
  However, this idea will be applied in a less familiar mode
  (``multiplicative'' rather than ``additive''),
  and so before giving a detailed proof
  we explain the intuition behind it
  making the simplifying assumption that $\EEE_1,\EEE_2,\ldots$ are superexpectations.

  By Lemma~\ref{lem:probability},
  it suffices to construct a positive supermartingale starting from $1$
  that is unbounded on $\omega\in\Omega$
  for which (\ref{eq:goal}) is not true.
  (We say that a supermartingale is unbounded on a sequence $\omega$
  if it is unbounded on its prefixes $\omega^n$ as $n\to\infty$.)
  Without loss of generality we will assume $\xi$ to be positive
  (we can always redefine $\xi:=\xi-\inf\xi$).
  According to Lemma \ref{lem:referee-4} (last statement),
  we can, without loss of generality,
  replace ``for which (\ref{eq:goal}) is not true'' by
  \begin{equation}\label{eq:event}
    \liminf_{n\to\infty}
    \UpperExpect(\xi\givn\omega^n)
    <
    a
    <
    b
    <
    \xi(\omega)
  \end{equation}
  where $a$ and $b$ are given positive rational numbers such that $a<b$.
  The supermartingale is defined as the capital process
  of the following strategy for Skeptic.
  Let $\omega\in\Omega$ be the sequence of moves
  chosen by World.
  Start with $1$ monetary unit.
  Wait until
  $\UpperExpect(\xi\givn\omega^n) < a$
  (if this never happens, do nothing, i.e., always choose constant
  $f_n\equiv\K_{n-1}$).
  As soon as this happens,
  choose a positive supermartingale $\SSS_1$ starting from $a$,
  $\SSS_1(\omega^n) = a$,
  whose upper limit $\psi\in\Omega\mapsto\limsup_{m\to\infty}\SSS_1(\psi^m)$
  exceeds $\xi$ on $\Gamma(\omega^n)$.
  Maintain capital $\SSS_1/a$ until $\SSS_1$ reaches a value $m_1>b$
  (at which point Skeptic's capital is $m_1/a>b/a$).
  After that do nothing until
  $\UpperExpect(\xi\givn\omega^n) < a$.
  As soon as this happens,
  choose a positive supermartingale $\SSS_2$ starting from $a$,
  $\SSS_2(\omega^n) = a$,
  whose upper limit exceeds $\xi$ on $\Gamma(\omega^n)$.
  Maintain capital $(m_1/a^2)\SSS_2$
  until $\SSS_2$ reaches a value $m_2>b$
  (at which point Skeptic's capital is $m_1m_2/a^2>(b/a)^2$).
  After that do nothing until
  $\UpperExpect(\xi\givn\omega^n) < a$.
  As soon as this happens,
  choose a positive supermartingale $\SSS_3$ starting from $a$
  whose upper limit exceeds $\xi$ on $\Gamma(\omega^n)$.
  Maintain capital $(m_1m_2/a^3)\SSS_3$
  until $\SSS_3$ reaches a value $m_3>b$
  (at which point Skeptic's capital is $m_1m_2m_3/a^3>(b/a)^3$).
  And so on.
  On the event (\ref{eq:event})
  Skeptic's capital will be unbounded.

  We start the formal proof
  by setting $\xi':=\xi-C$,
  where $C$ is any constant satisfying $C<\inf\xi$.
  Let $[a_i,b_i]$, $i=1,2,\ldots$,
  be an enumeration of all intervals with $0\le a_i<b_i<\infty$ and both end-points rational.
  For each $i$ we will define a positive supermartingale $\SSS^i$ with
  $
    \SSS^i(\Box) = 1
  $
  such that $\SSS^i(\omega^n)$ converges to $\infty$ as $n\to\infty$
  when
  \begin{equation}\label{eq:event-i}
    \liminf_{n\to\infty}
    \UpperExpect(\xi'\givn\omega^n)
    <
    a_i
    <
    b_i
    <
    \xi'(\omega)
  \end{equation}
  (cf.\ (\ref{eq:event})).
  First we define two sequences of cuts
  $\sigma^i_0,\sigma^i_1,\sigma^i_2,\ldots$ and $\tau^i_1,\tau^i_2,\ldots$
  and a family of supermartingales $\SSS_s$, $s\in\cup_{k=1}^{\infty}\tau^i_k$
  (the dependence of $\SSS_s$ on $i$ is not indicated explicitly
  but should always be borne in mind).
  Set $\sigma^i_0:=\{\Box\}$.
  For $k=1,2,\ldots$,
  set
  \begin{equation*}
    \tau^i_k
    :=
    \{s \st \UpperExpect(\xi'\givn s)<a_i, \exists t\subset s: t\in\sigma^i_{k-1},
      \forall u\in(t,s): \UpperExpect(\xi'\givn u)\ge a_i\},
  \end{equation*}
  choose for each $s\in\tau^i_k$ a positive supermartingale $\SSS_s$
  satisfying $\SSS_s(s)<a_i$ and, for all $\omega\in\Gamma(s)$,
  $\liminf_n\SSS_s(\omega^n)\ge\xi'(\omega)$,
  and set
  \begin{equation*}
    \sigma^i_k
    :=
    \{
      s
      \st
      \exists t\subset s: t\in\tau^i_{k},
      \SSS_t(s)>b_i,
      \forall u\in(t,s): \SSS_t(u)\le b_i
    \}.
  \end{equation*}
  This definition is inductive:
  the two cuts and the family of supermartingales are defined in the indicated order:
  first $\sigma^i_0$,
  then $\tau^i_1$, then $\SSS_s$ for $s\in\tau^i_1$, then $\sigma^i_1$,
  then $\tau^i_2$, then $\SSS_s$ for $s\in\tau^i_2$, then $\sigma^i_2$,
  etc.
  Now we define $\SSS^i$ inductively.
  Set $\SSS^i(\Box):=1$.
  For all situations $s$ and all $x\in\XXX$,
  define $\SSS^i(sx)$ via $\SSS^i(s)$ as follows:
  \begin{itemize}
  \item
    First suppose that $\SSS^i(s)<\infty$ and,
    for some $k\in\bbbn$, $s\in[\tau^i_k,\sigma^i_k)$.
    Let $k$ be the unique value satisfying $s\in[\tau^i_k,\sigma^i_k)$.
    Set $t:=s^{\tau^i_k}$
    and
    $\SSS^i(sx) := \SSS^i(s) \SSS_t(sx) / \SSS_t(s)$
    (notice that, by induction, $\SSS^i(s)<\infty$ implies $\SSS_t(s)<\infty$).
  \item
    Otherwise,
    set $\SSS^i(sx):=\SSS^i(s)$.
  \end{itemize}
  Since each $\SSS_s$ is a supermartingale
  (strictly positive, by Lemma~\ref{lem:referee-1}),
  Axiom~\ref{ax:scaling} shows that $\SSS^i$ is also a supermartingale.
  It is positive by construction.

  Let us check that each supermartingale $\SSS^i$
  satisfies $\limsup_n\SSS^i(\omega^n)=\infty$
  for $\omega\in\Omega$ satisfying~(\ref{eq:event-i}).
  For all $k\in\bbbn$ and all $\omega$ satisfying~(\ref{eq:event-i}),
  $\omega^{\sigma^i_k}$ exists and satisfies
  \begin{equation}\label{eq:SSS}
    \SSS^i(\omega^{\sigma^i_k})
    =
    \frac{\SSS_{\omega^{\tau^i_1}}(\omega^{\sigma^i_1})}
         {\SSS_{\omega^{\tau^i_1}}(\omega^{\tau^i_1})}\,
    \frac{\SSS_{\omega^{\tau^i_2}}(\omega^{\sigma^i_2})}
         {\SSS_{\omega^{\tau^i_2}}(\omega^{\tau^i_2})}
    \cdots
    \frac{\SSS_{\omega^{\tau^i_k}}(\omega^{\sigma^i_k})}
         {\SSS_{\omega^{\tau^i_k}}(\omega^{\tau^i_k})}
    \ge
    (b_i/a_i)^k
    \to\infty
  \end{equation}
  as $k\to\infty$.
  Setting
  \begin{equation*}
    \TTT
    :=
    \sum_{i=1}^{\infty}
    2^{-i}
    \SSS^i
  \end{equation*}
  and assuming that $\EEE_1,\EEE_2,\ldots$ are superexpectations,
  we obtain a positive supermartingale $\TTT$ with $\TTT(\Box)=1$
  that is unbounded on the complement of~(\ref{eq:goal}).
  Application of Lemma~\ref{lem:probability} completes the proof
  under the assumption that $\EEE_1,\EEE_2,\ldots$ are superexpectations.

  It remains to get rid of the assumption.
  To this end we modify the definition of the supermartingales $\SSS_s$:
  now in each situation $s\in\XXX^*$
  such that $\UpperExpect(\xi'\givn s)<\infty$
  we fix a strictly positive supermartingale $\SSS_s$ such that
  $\SSS_s(s)<\UpperExpect(\xi'\givn s)+2^{-\lvert s\rvert}$
  and, for all $\omega\in\Gamma(s)$,
  $\liminf_n\SSS_s(\omega^n)\ge\xi'(\omega)$.
  This definition does not depend on $i$ anymore.
  Using the new definition of $\SSS_s$,
  define stopping times $\sigma^i_k,\tau^i_k$,
  supermartingales $\SSS^i$, and a function $\TTT$ as before;
  remember that in the definition of $\SSS^i$,
  $\SSS^i(sx):=\infty$ whenever $\SSS^i(s)=\infty$.
  For each $i\in\bbbn$ set
  $$
    A_i
    :=
    \left\{
      s\in\XXX^*
      \st
      \SSS^i(s)<\infty,
      \exists k\in\bbbn: s\in[\tau^i_k,\sigma^i_k)
    \right\}
  $$
  (this is the set of situations in which $\SSS^i$ is ``active''),
  and for each $s\in A_i$ set $T(s,i):=s^{\tau^i_k}$,
  where $k$ satisfies $s\in[\tau^i_k,\sigma^i_k)$
  (there is only one such $k$).

  Fix an arbitrary situation $s\in\XXX^{n-1}$, for some $n\in\bbbn$;
  our next goal is to prove $\EEE_{n}(\TTT(s\,\cdot))\le\TTT(s)$.
  Without loss of generality,
  assume $\TTT(s)<\infty$.
  Setting, for $t\subseteq s$,
  $$
    w_t
    :=
    \sum_{i\in\bbbn:s\in A_i,T(s,i)=t}
    2^{-i} \SSS^i(s)
    \in
    [0,\infty],
  $$
  we have
  \begin{multline*}
    \TTT(sx)
    =
    \sum_{i=1}^\infty
    2^{-i} \SSS^i(sx)
    =
    \sum_{i:s\in A_i}
    2^{-i} \SSS^i(s)
    \frac{\SSS_{T(s,i)}(sx)}{\SSS_{T(s,i)}(s)}
    +
    \sum_{i:s\notin A_i}
    2^{-i} \SSS^i(s)\\
    =
    \sum_{t\subseteq s}
    w_t
    \frac{\SSS_t(sx)}{\SSS_t(s)}
    +
    \sum_{i:s\notin A_i}
    2^{-i} \SSS^i(s)
  \end{multline*}
  (the denominators $\SSS_{T(s,i)}(s)$ and $\SSS_t(s)$
  are finite because of our assumption $\TTT(s)<\infty$),
  which implies, since the sum $\sum_{t\subseteq s}$ is finite
  (namely, contains $\lvert s\rvert+1$ addends),
  \begin{multline*}
    \EEE_{n}(\TTT(s\,\cdot))
    =
    \sum_{t\subseteq s}
    w_t
    \frac{\EEE_{n}(\SSS_t(s\,\cdot))}{\SSS_t(s)}
    +
    \sum_{i:s\notin A_i}
    2^{-i} \SSS^i(s)\\
    \le
    \sum_{t\subseteq s}
    w_t
    \frac{\SSS_t(s)}{\SSS_t(s)}
    +
    \sum_{i:s\notin A_i}
    2^{-i} \SSS^i(s)
    =
    \TTT(s).
  \end{multline*}
  Therefore, $\TTT$ is a supermartingale.

  Notice that~(\ref{eq:SSS}) will still hold
  for $\omega$ satisfying~(\ref{eq:event-i})
  if we replace $(b_i/a_i)^k$ by
  \begin{equation}\label{eq:product}
    \prod_{j=1}^k
    \frac{b_i}{a_i+2^{-\lvert\omega^{\tau^i_j}\rvert}};
  \end{equation}
  since $\lvert\omega^{\tau^i_j}\rvert\to\infty$ as $j\to\infty$,
  the product (\ref{eq:product}) still tends to $\infty$ as $k\to\infty$.
  Therefore, $\TTT$ is unbounded on the complement of~(\ref{eq:goal}).
\end{proof}

In the rest of the article we will derive a series of corollaries
from Theorem~\ref{thm:levy}.
First of all,
specializing Theorem~\ref{thm:levy} to the indicators of events,
we obtain:
\begin{corollary}\label{cor:levy}
  Let $E$ be any event.
  For almost all $\omega\in E$,
  \begin{equation*} 
    \UpperProb(E\givn\omega^n)
    \to
    1
  \end{equation*}
  as $n\to\infty$.
\end{corollary}

It is easy to check that we cannot replace the $\ge$ in (\ref{eq:goal}) by $=$,
even when $\xi$ is the indicator of an event.
For example,
suppose that $\XXX=\{0,1\}$ and each $\EEE_n$ is the $\sup$ functional:
$\EEE_n(f):=\sup_{x\in\XXX}f(x)$ for all $n\in\bbbn$ and $f\in\bbbrbarXXX$.
If $E$ consists of binary sequences containing only finitely many $1$s,
$\UpperProb(E\givn\omega^n)=1$ for all $\omega$ and $n$;
therefore,
\begin{equation*}
  \liminf_{n\to\infty}
  \UpperProb(E\givn\omega^n)
  \ne
  \III_E(\omega)
\end{equation*}
for all $\omega\in E^c$,
and $\UpperProb(E^c)=1$.

\subsection*{The case of a determinate expectation or probability}
\ifWP
  \addcontentsline{toc}{subsection}%
  {The case of a determinate expectation or probability}
\fi

\ifnotJOURNAL
  In Section~41 of \cite{levy:1937} (pp.~128--130),
  L\'evy states his zero-one law in terms of a property $E$
  that a sequence $X_1,X_2,\dots$ of random variables might or might not have.
  He writes $\text{Pr.}\{E\}$ for the initial probability of $E$,
  and $\text{Pr}_n\{E\}$ for its probability after $X_1,\dots,X_n$ is known.
  He remarks that if $\text{Pr.}\{E\}$ is well defined
  (i.e., if $E$ is measurable),
  then the conditional probabilities $\text{Pr}_n\{E\}$ are also well defined.
  Then he states the law as follows (our translation from the French):
  \begin{quotation}
    Except in cases that have probability zero,
    if $\text{Pr.}\{E\}$ is determined,
    then $\text{Pr}_n\{E\}$ tends, as $n$ tends to infinity, to one
    if the sequence $X_1,X_2,\dots$ verifies the property $E$,
    and to zero in the contrary case.
  \end{quotation}
  In this subsection we will derive a game-theoretic result
  that resembles L\'evy's statement of his result.
  We will be concerned with functions $\xi$ satisfying
  $\UpperExpect(\xi)=\LowerExpect(\xi)$
  and events $E$ satisfying
  $\UpperProb(E)=\LowerProb(E)$.
\fi
\begin{lemma}\label{lem:determinate}
  Suppose $\EEE_1,\EEE_2,\ldots$ are superexpectations.
  Let a function $\xi:\Omega\to\bbbr$ satisfy
  $\UpperExpect(\xi)=\LowerExpect(\xi)\in\bbbr$.
  Then it is almost certain that it also satisfies
  $
    \UpperExpect(\xi\givn\omega^n)
    =
    \LowerExpect(\xi\givn\omega^n)
  $
  for all $n$.
\end{lemma}
\begin{proof}
  For any strictly positive $\epsilon$,
  there exist bounded below supermartingales $\SSS_1$ and $\SSS_2$ such that
  $$
    \SSS_1(\Box)<\UpperExpect(\xi) + \epsilon/2,
    \quad
    \SSS_2(\Box)<\UpperExpect(-\xi) + \epsilon/2
  $$
  and, for all $\omega\in\Omega$,
  $$
    \liminf_{n\to\infty}\SSS_1(\omega^n)\ge\xi(\omega),
    \quad
    \liminf_{n\to\infty}\SSS_2(\omega^n)\ge-\xi(\omega).
  $$
  Set $\SSS := \SSS_1 + \SSS_2$.
  The assumption $\UpperExpect(\xi)=\LowerExpect(\xi)\in\bbbr$
  can also be written $\UpperExpect(\xi)+\UpperExpect(-\xi)=0$.
  So the supermartingale $\SSS$ satisfies $\SSS(\Box)<\epsilon$
  and $\liminf_{n\to\infty}\SSS(\omega^n)\ge0$ for all $\omega\in\Omega$;
  by Lemma~\ref{lem:referee-1}, the supermartingale $\SSS$ is positive.

  Fix $n$ and $\delta > 0$, and let $E$ be the event that
  \begin{equation*}
    \UpperExpect(\xi\givn\omega^n)
    +
    \UpperExpect(-\xi\givn\omega^n)
    >
    \delta.
  \end{equation*}
  By the definition of conditional upper expectation,
  \begin{equation*}
    \SSS_1(\omega^n) \ge \UpperExpect(\xi\givn\omega^n)
  \quad\text{and}\quad
    \SSS_2(\omega^n) \ge \UpperExpect(-\xi\givn\omega^n).
  \end{equation*}
  So $\SSS(\omega^n)>\delta$ for all $\omega\in E$.
  So, by Lemma~\ref{lem:probability},
  the upper probability of $E$ is less than $\epsilon/\delta$.
  Since $\epsilon$ may be as small as we like for fixed $\delta$,
  this shows that $E$ has upper probability zero.
  Letting $\delta$ range over the strictly positive rational numbers
  and $n$ over $\{0,1,2,\ldots\}$
  and applying the last part of Lemma \ref{lem:referee-4},
  we obtain the statement of the lemma.
\end{proof}

\begin{corollary}
  Suppose $\EEE_1,\EEE_2,\ldots$ are superexpectations.
  Let $\xi:\Omega\to\bbbr$ be a bounded function
  for which $\UpperExpect(\xi)=\LowerExpect(\xi)$.
  Then, almost surely,
  $
    \UpperExpect(\xi\givn\omega^n)
    =
    \LowerExpect(\xi\givn\omega^n)
    \to
    \xi(\omega)
  $
  as $n\to\infty$.
\end{corollary}
\begin{proof}
  By Theorem \ref{thm:levy},
  \begin{equation*}
    \liminf_{n\to\infty}\UpperExpect(\xi\givn\omega^n)
    \ge
    \xi(\omega)
  \end{equation*}
  for almost all $\omega\in\Omega$
  and (applying the theorem to $-\xi$)
  \begin{equation*}
    \limsup_{n\to\infty}\LowerExpect(\xi\givn\omega^n)
    \le
    \xi(\omega)
  \end{equation*}
  for almost all $\omega\in\Omega$.
\end{proof}

Our definitions (\ref{eq:probability}) make it easy to obtain
the following corollary for events.
\begin{corollary}
  Suppose $\EEE_1,\EEE_2,\ldots$ are superexpectations.
  Let $E$ be an event for which $\UpperProb(E)=\LowerProb(E)$.
  Then, almost surely,
  $
    \UpperProb(E\givn\omega^n)
    =
    \LowerProb(E\givn\omega^n)
    \to
    \III_E
  $ as $n\to\infty$.
\end{corollary}

\section{An implication for the foundations of game-theoretic probability theory}
\label{sec:implication}

Let  $\xi:\Omega\to\bbbr$ be a bounded function,
and let $s:=\Box$.
We will obtain an equivalent definition of the upper expectation
$\UpperExpect(\xi\givn s)=\UpperExpect(\xi)$
if we replace the phrase ``for all $\omega\in\Gamma(s)$''
in (\ref{eq:upper-expectation})
by ``for almost all $\omega\in\Omega$''.
It turns out that if we do so,
the infimum in (\ref{eq:upper-expectation}) becomes attained;
namely,
it is attained by the supermartingale
$\SSS(s):=\UpperExpect(\xi\givn s)$, $s\in\XXX^*$.
(This fact is the key technical tool used in \cite{vovk/shen:2010-local}.)
In view of Theorem \ref{thm:levy},
to prove this statement it suffices to check that
$\SSS$ is indeed a supermartingale
(it will be bounded below by Lemma~\ref{lem:referee-1}).
We will prove a slightly stronger statement.
\begin{lemma}\label{lem:martingale}
  Let $\xi:\Omega\to\bbbrbar$.
  Then $\SSS:=\UpperExpect(\xi\givn\cdot)$ is a supermartingale.
\end{lemma}
\begin{proof}
  Let $n\in\bbbn$, $s\in\XXX^{n-1}$, and $r>\SSS(s)$.
  By Lemma~\ref{lem:referee-2},
  there exists a bounded below supermartingale $\TTT$
  such that $\TTT(s)<r$ and $\liminf_n\TTT(\omega^n)\ge\xi(\omega)$
  for all $\omega\in\Omega$.
  Then we have $\SSS(sx)\le\TTT(sx)$ for all $x\in\XXX$,
  and so we have
  $$
    \EEE_n(\SSS(s\,\cdot))
    \le
    \EEE_n(\TTT(s\,\cdot))
    \le
    \TTT(s)
    <
    r.
  $$
  Letting $r\to\SSS(s)$ (if $\SSS(s)<\infty$)
  shows that $\EEE_n(\SSS(s\,\cdot))\le\SSS(s)$,
  which proves the supermartingale property.
\end{proof}

The following simple example shows that replacing ``for all $\omega\in\Gamma(s)$''
in (\ref{eq:upper-expectation})
by ``for almost all $\omega\in\Omega$''
is essential if we want the infimum to be attained.
\begin{example}\label{ex:not-attained}
  Consider the coin-tossing protocol,
  as in Example \ref{ex:expectation}.
  Let $A$ be the set of all $\omega\in\Omega$
  containing only finitely many $1$s,
  let $\xi:=\III_A$,
  and let $s:=\Box$.
  The infimum in (\ref{eq:upper-expectation}) is not attained:
  there exist no supermartingale $\SSS$
  satisfying $\SSS(\Box)=\UpperExpect(\xi)$
  and $\liminf_{n\to\infty}\SSS(\omega^n)\ge\xi(\omega)$
  for all $\omega\in\Omega$.
\end{example}
\begin{proof}
  By Lemma~\ref{lem:referee-1}, such an $\SSS$ would be positive.
  Let $L$ be the uniform probability measure on $\{0,1\}^{\infty}$
  equipped with the Borel $\sigma$-algebra.
  Since $\UpperProb(E)=L(E)$ for all Borel sets in $\{0,1\}^{\infty}$
  (\cite{shafer/vovk:2001}, Proposition~8.5),
  we would have $\SSS(\Box)=\UpperExpect(\xi)=0$.
  A positive supermartingale with initial value $0$
  in the coin-tossing protocol must be a constant.
\end{proof}

\section{More explicit zero-one laws}
\label{sec:explicit}

In this section we will deduce two corollaries
from our game-theoretic version of L\'evy's zero-one law:
Kolmogorov's zero-one law and the ergodicity of Bernoulli shifts.
Both corollaries were proved in \cite{\Takemura} directly.
These two results are more general
than the corresponding measure-theoretic results;
see \cite{shafer/vovk:2001}, Section 8.1,
for relations between measure-theoretic results
and their game-theoretic counterparts.

For each $N\in\bbbn$,
let $\FFF_N$ be the set of all events $E$
that are properties of $(\omega_{N},\omega_{N+1},\ldots)$ only
(i.e., $E$ such that, for all $\omega,\omega'\in\Omega$,
$\omega'\in E$ whenever $\omega\in E$ and $\omega'_n=\omega_n$ for all $n\ge N$).
By a \emph{tail event} we mean an element of $\cap_N\FFF_N$.
In other words, an event $E\subseteq\Omega$ is a tail event
if any sequence in $\Omega$ that agrees from some point onwards
with a sequence in $E$ is also in $E$.

\subsection*{Kolmogorov's zero-one law}
\ifWP
  \addcontentsline{toc}{subsection}%
  {Kolmogorov's zero-one law}
\fi

L\'evy's zero-one law immediately implies the following game-theoretic version
of Kolmogorov's zero-one law.
\begin{corollary}[\cite{\Takemura}]\label{cor:kolmogorov-1}
  For all tail events $E\subseteq\Omega$,
  $\UpperProb(E)\in\{0,1\}$.
\end{corollary}
\begin{proof}
  First we will check that,
  for each $N\in\bbbn$ and each $E\in\FFF_{N}$,
  $\UpperProb(E\givn s)$
  does not depend on $s\in\XXX^{N-1}$.
  Indeed, let $s,t\in\XXX^{N-1}$
  and $\UpperProb(E\givn s)<r$.
  Choose a bounded below supermartingale $\SSS$ such that
  $\SSS(s)<r$ and $\liminf_n\SSS(\omega^n)\ge\III_E(\omega)$
  for all $\omega\in\Gamma(s)$.
  We will write $ab$ for the concatenation of two situations $a\in\XXX^*$ and $b\in\XXX^*$,
  and $a\omega$ for the concatenation of $a\in\XXX^*$ and $\omega\in\Omega$.
  The supermartingale
  $$
    \SSS'(u)
    :=
    \begin{cases}
      \SSS(sv) & \text{if $u=tv$ for some (uniquely determined) $v\in\XXX^*$}\\
      \infty & \text{otherwise}
    \end{cases}
  $$
  witnesses that $\UpperProb(E\givn t)<r$,
  in the sense that $\SSS'(t)=\SSS(s)<r$
  and, for all $\omega\in\Omega$,
  $$
    \liminf_n \SSS'(t\omega^n)
    =
    \liminf_n \SSS(s\omega^n)
    \ge
    \III_E(s\omega)
    =
    \III_E(t\omega).
  $$
  Since this is true for all $s,t\in\XXX^{N-1}$,
  $\UpperProb(E\givn s)$ cannot depend on $s\in\XXX^{N-1}$.

  By Lemma~\ref{lem:martingale} and Axiom~\ref{ax:norm},
  this implies $\UpperProb(E\givn\omega^{N-1})=\UpperProb(E)$
  for $N\in\bbbn$, $E\in\FFF_N$, and $\omega\in\Omega$.
  By Corollary~\ref{cor:levy},
  for $E\in\FFF$ we have $\UpperProb(E)=1$ for almost all $\omega\in E$,
  which is equivalent to $\UpperProb(E)\in\{0,1\}$.
\end{proof}

We say that an event $E$ is \emph{fully unprobabilized}
if $\LowerProb(E)=0$ and $\UpperProb(E)=1$.
Since complements of tail events are also tail events,
we obtain the following corollary to Corollary \ref{cor:kolmogorov-1}.
\begin{corollary}[\cite{\Takemura}]\label{cor:kolmogorov-2}
  If $E\subseteq\Omega$ is a tail event,
  then $E$ is almost certain, almost impossible,
  or fully unprobabilized.
\end{corollary}

\subsection*{Ergodicity of Bernoulli shifts}
\ifWP
  \addcontentsline{toc}{subsection}%
  {Ergodicity of Bernoulli shifts}
\fi

In this subsection we consider a special case of Protocol~1
where $\EEE_1=\EEE_2=\cdots$.
We write $\theta$ for the shift operator,
which deletes the first element from a sequence in $\XXX^{\infty}$:
\begin{equation*}
  \theta: x_1 x_2 x_3 \ldots \mapsto x_2 x_3 \mathenddots
\end{equation*}
We call an event $E\subseteq\Omega$ \emph{weakly invariant}
if $\theta E \subseteq E$.
In accordance with standard terminology,
an event $E$ is \emph{invariant} if $E=\theta^{-1}E$.
\begin{lemma}\label{lem:weakly-invariant}
  An event $E$ is invariant if and only if both $E$ and $E^c$ are weakly invariant.
\end{lemma}
\begin{proof}
  We will give the simple argument from \cite{\Takemura}.
  If $E$ is invariant, then $E^c$ is also invariant,
  because the inverse map commutes with complementation.
  Hence in this case both $E$ and $E^c$ are weakly invariant.

  Conversely suppose that $\theta E \subseteq E$ and $\theta E^c \subseteq E^c$.
  The first inclusion is equivalent to $E\subseteq\theta^{-1}E$
  and the second is equivalent to $E^c\subseteq\theta^{-1}E^c$.
  Since the right-hand sides of the last two inclusions
  are disjoint,
  these inclusions are in fact equalities.
\end{proof}

The following corollary asserts the ergodicity of Bernoulli shifts.
\begin{corollary}[\cite{\Takemura}]
  \label{cor:weakly-invariant}
  Suppose $\EEE_1=\EEE_2=\cdots$.
  For all weakly invariant events $E$,
  $\UpperProb(E)\in\{0,1\}$.
\end{corollary}
\begin{proof}
  For any weakly invariant event $E$ and any situation $s$,
  $\UpperProb(E\givn s)\le\UpperProb(E)$.
  Indeed, let $\UpperProb(E)<r$.
  Choose a bounded below supermartingale $\SSS$
  such that $\SSS(\Box)<r$ and
  $\liminf_n\SSS(\omega^n)\ge\III_E(\omega)$ for all $\omega\in\Omega$.
  Define a new bounded below supermartingale $\SSS'$ by
  $\SSS'(st):=\SSS(t)$ for all $t\in\XXX^*$
  and $\SSS'(t):=\infty$ for all $t\in\XXX^*$ such that $s\not\subseteq t$.
  This supermartingale witnesses that $\UpperProb(E\givn s)<r$,
  in the sense that
  $\SSS'(s)=\SSS(\Box)<r$ and
  $$
    \liminf_n \SSS'(s\omega^n)
    =
    \liminf_n \SSS(\omega^n)
    \ge
    \III_E(\omega)
    \ge
    \III_E(s\omega),
    \quad
    \forall\omega\in\Omega,
  $$
  the last inequality following from $s\omega\in E\Rightarrow\omega\in E$.

  Therefore, we have $\UpperProb(E\givn\omega^n)\le\UpperProb(E)$
  when $E$ is weakly invariant.
  By Corollary~\ref{cor:levy},
  for almost all $\omega\in E$ it is true that
  $\UpperProb(E)=1$.
  Therefore, $\UpperProb(E)$ is either $0$ or $1$.
\end{proof}

In view of Lemma \ref{lem:weakly-invariant}
we obtain the following corollary to Corollary \ref{cor:weakly-invariant}.
\begin{corollary}[\cite{\Takemura}]\label{cor:invariant}
  Suppose $\EEE_1=\EEE_2=\cdots$.
  If $E$ is an invariant event,
  then $E$ is almost certain,
  almost impossible,
  or fully unprobabilized.
\end{corollary}
\noindent
Since each invariant event is a tail event,
Corollary~\ref{cor:invariant} also follows from Corollary~\ref{cor:kolmogorov-2}.

\section{The generality of the basic prediction protocol}
\label{sec:generality}

Let $\mathbf{E}(X)$ be the set of all outer probability contents on a set $X$.
Protocol 1 is a special case of the following apparently more general protocol.

\medskip

\noindent
\textsc{Protocol 2. Prediction protocol with Forecaster}

\noindent
\textbf{Parameters:} non-empty set $\XXX$, non-empty sets $\PPP_1,\PPP_2,\ldots$,\\
\hspace*{\IndentII}and function $\EEE:p\in\cup_n\PPP_n\mapsto\EEE_p\in\mathbf{E}(\XXX)$

\noindent
\textbf{Protocol:}

\parshape=7
\IndentI   \WidthI
\IndentI   \WidthI
\IndentII  \WidthII
\IndentII  \WidthII
\IndentII  \WidthII
\IndentII  \WidthII
\IndentI   \WidthI
\noindent
Skeptic announces $\K_0\in\bbbrbar$.\\
FOR $n=1,2,\ldots$:\\
  Forecaster announces $p_n\in\PPP_n$.\\
  Skeptic announces $f_n$ such that $\EEE_{p_n}(f_n)\le\K_{n-1}$.\\
  Reality announces $x_n\in\XXX$.\\
  $\K_n := f_n(x_n)$.\\
END FOR

\medskip

\noindent
As compared with Protocol~1,
Protocol~2 involves another player, Forecaster;
World is now called Reality.
(We will see later that another interpretation
is that World is split into two players:
Reality and Forecaster;
cf.\ \cite{shafer/vovk:2001}, p.~90.)
At the beginning of each trial
Forecaster gives his prediction $p_n$ for Reality's move $x_n$;
the prediction is chosen from a set $\PPP_n$,
the \emph{prediction space} for trial $n$.
We will use the notation $\PPP$ for $\cup_n\PPP_n$.
After Forecaster's move Skeptic chooses a gamble,
which we represent as a function $f_n$ on $\XXX$:
$f_n(x)$ is the payoff of the gamble
if Reality chooses $x$ as the trial's outcome.
The gambles available to Skeptic are determined
by Forecaster's prediction (via the function $\EEE:\PPP\to\mathbf{E}(\XXX)$).

Protocol~1 is a special case of Protocol~2
obtained by taking distinct one-element sets
$\PPP_1,\PPP_2,\mathenddots$\endsentence\
In some sense Protocol~1 describes independent trials
(since $\EEE_1,\EEE_2,\ldots$ are given in advance)
whereas Protocol~2 describes dependent trials
(Forecaster has a say in choosing $\EEE_{p_1},\EEE_{p_2},\ldots$).

\begin{remark}
  Another version of the prediction protocol with Forecaster
  is where Forecaster chooses
  the superexpectation functional directly.
  This is a special case of our protocol
  with $\PPP_n=\mathbf{E}(\XXX)$ for all $n$
  and with $\EEE:\PPP\to\mathbf{E}(\XXX)$ the identity function.
  The reader will also notice that allowing $\EEE$
  to depend not only on Forecaster's last move
  but also on his and Reality's previous moves
  is straightforward but does not lead to stronger results:
  the seemingly more general results easily follow from our results.
\end{remark}

We call the set
$\Omega:=\prod_{n=1}^{\infty}(\PPP_n\times\XXX)$
of all infinite sequences of Forecaster's and Reality's moves
the \emph{sample space}.
The elements of the set
$\bigcup_{n=0}^{\infty}\prod_{i=1}^n(\PPP_i\times\XXX)$
of all finite sequences of Forecaster's and Reality's moves
are called \emph{clearing situations},
and the elements of the set
$
  \bigcup_{n=0}^{\infty}
  \left(
    \prod_{i=1}^n(\PPP_i\times\XXX)
    \times
    \PPP_{n+1}
  \right)
$
are called \emph{betting situations}.
We will be mostly interested in clearing situations.
For each clearing situation $s$ we let $\Gamma(s)\subseteq\Omega$
stand for the set of all infinite extensions in $\Omega$ of $s$
and let $\Box$ be the empty clearing situation.

The \emph{level} $\lvert s\rvert$ of a clearing situation $s$
is the number of predictions in $s$.
In other words, $n$ is the level
of clearing situations of the form $p_1x_1\ldots p_nx_n$.
If $\omega\in\Omega$ and $n\in\{0,1,\ldots\}$,
$\omega^n$ is defined to be the unique clearing situation of level $n$
that is a prefix of $\omega$.


A function $\SSS$ defined on the clearing situations
and taking values in $\bbbrbar$  is called a \emph{supermartingale}
if,
for each $n\in\bbbn$, each clearing situation $s$ at level $n-1$,
and each $p\in\PPP_n$,
$$
  \EEE_p(\SSS(sp\,\cdot)) \le \SSS(s).
$$

\noindent
For each function $\xi:\Omega\to\bbbrbar$
and each clearing situation $s$,
we define the (conditional) \emph{upper expectation} of $\xi$ given $s$
by the same formula (\ref{eq:upper-expectation}),
where $\SSS$ ranges over the supermartingales that are bounded below,
and we define the \emph{lower expectation} of $\xi$ given $s$
by (\ref{eq:lower-expectation}).
As before,
upper and lower probabilities of sets are defined by (\ref{eq:probability}).

\begin{lemma}\label{lem:levy-general}
  Theorem~\ref{thm:levy} and, therefore, Corollary~\ref{cor:levy}
  continue to hold under the definitions of this section.
\end{lemma}
\begin{proof}
  Protocol~2 can be embedded in Protocol~1 as follows.
  Let the parameters of Protocol~2 be $\XXX$, $\PPP_1,\PPP_2,\ldots$, and $\EEE$;
  as before, $\PPP:=\cup_n\PPP_n$.
  For each $n\in\bbbn$, define an outer probability content $\EEE_n$
  on $\XXX':=\PPP\times\XXX$ by
  $$
    \EEE_n(f)
    :=
    \sup_{p\in\PPP_n}
    \EEE_p(f(p,\cdot)),
    \quad
    f:\XXX'\to\bbbrbar.
  $$
  (Axioms~\ref{ax:order}--\ref{ax:norm} are easy to check for $\EEE_n$;
  e.g., Axiom~\ref{ax:sum} follows from $\sup(f+g)\le\sup f+\sup g$.)
  The parameters of Protocol~1 will be $\XXX'$ and $\EEE_1,\EEE_2,\mathenddots$

  Our goal is to prove (\ref{eq:goal}) in Protocol~2.
  Let $\xi:\Omega\to(-\infty,\infty]$ be bounded below.
  To each $\omega$ in the sample space $\Omega$ of Protocol~2
  corresponds the same sequence
  in the sample space $\Omega':=(\XXX')^{\infty}$ of Protocol~1;
  therefore, $\Omega\subseteq\Omega'$
  (perhaps $\Omega\subset\Omega'$).
  Let $\xi':\Omega'\to(-\infty,\infty]$ be the extension of $\xi$
  defined by, say, $\xi'(\omega):=0$ for $\omega\notin\Omega$.
  Since the analogue
  $$
    \liminf_{n\to\infty}
    \UpperExpect(\xi'\givn\omega^n)
    \ge
    \xi'(\omega),
    \text{ for almost all }
    \omega\in\Omega',
  $$
  of (\ref{eq:goal}) holds in Protocol~1,
  we are only required to prove two statements:
  \begin{enumerate}
  \item 
    $\UpperExpect(\xi\givn\omega^n) \ge \UpperExpect(\xi'\givn\omega^n)$
    for all $\omega\in\Omega$ and $n=0,1,\ldots$,
    where the $\UpperExpect$ on the left-hand side refers to Protocol~2
    and the $\UpperExpect$ on the right-hand side refers to Protocol~1.
  \item 
    If an event $E\subseteq\Omega'$ is null in Protocol~1,
    $E\cap\Omega$ will be null in Protocol~2.
  \end{enumerate}

  First we prove Statement~1.
  Fix $\omega\in\Omega$ and $n\in\{0,1,\ldots\}$.
  For each $\epsilon>0$
  there is a bounded below supermartingale $\SSS$ in Protocol~2 such that
  $\SSS(\omega^n) \le \UpperExpect(\xi\givn\omega^n) + \epsilon$
  and $\liminf_{n\to\infty} \SSS(\omega^n\psi) \ge \xi(\omega^n\psi)$
  for all $\psi\in\Omega$.
  Let $\SSS'$ be the extension of $\SSS$ to $(\XXX')^*$
  defined as $\infty$ on the situations in Protocol~1
  that are not clearing situations in Protocol~2.
  By the definition of $\EEE_n$,
  $\SSS'$ will be a supermartingale in Protocol~1:
  if $s$ is a clearing situation in Protocol~2
  (the case where it is not is trivial),
  we have
  $$
    \EEE_{n}(\SSS(s\,\cdot))
    =
    \sup_{p\in\PPP_n}
    \EEE_p(\SSS(sp\,\cdot))
    \le
    \sup_{p\in\PPP_n}
    \SSS(s)
    =
    \SSS(s),
  $$
  where $n=\lvert s\rvert+1$.
  The supermartingale $\SSS'$ witnesses that
  $\UpperExpect(\xi'\givn\omega^n) \le \UpperExpect(\xi\givn\omega^n) + \epsilon$:
  indeed,
  we have
  $
    \liminf_{m\to\infty} \SSS'(\omega^n\psi^m)
    =
    \liminf_{m\to\infty} \SSS(\omega^n\psi^m)
    \ge
    \xi(\omega^n\psi)
    =
    \xi'(\omega^n\psi)
  $
  for $\psi\in\Omega$
  since $\SSS'$ is an extension of $\SSS$ and $\xi'$ is an extension of $\xi$,
  and we have
  $\liminf_{m\to\infty} \SSS'(\omega^n\psi^m) = \infty \ge 0 = \xi'(\omega^n\psi)$
  for $\psi\in\Omega'\setminus\Omega$
  since in this case $\omega^n\psi^m$
  is not a clearing situation in Protocol~2 from some $m$ on.
  Setting $\epsilon\to0$ completes the proof of Statement~1.

  To prove Statement~2,
  it suffices to check that for any supermartingale $\SSS'$ in Protocol~1
  its restriction to the clearing situations in Protocol~2
  will be a supermartingale in Protocol~2.
  This follows immediately from the definition of $\EEE_1,\EEE_2,\ldots$:
  if $n\in\bbbn$, $s$ is a clearing situation at level $n-1$, and $p\in\PPP_n$,
  $$
    \EEE_p
    \left(
      \SSS'(sp\,\cdot)
    \right)
    \le
    \sup_{p\in\PPP_n}
    \EEE_p
    \left(
      \SSS'(sp\,\cdot)
    \right)
    =
    \EEE_n
    \left(
      \SSS'(s\,\cdot)
    \right)
    \le
    \SSS'(s)
  $$
  (the first two $\cdot$ stand for an element of $\XXX$
  and the last $\cdot$ stands for an element of $\XXX'$).
\end{proof}

Notice that Theorem~\ref{thm:levy}
is a special case of Lemma~\ref{lem:levy-general},
corresponding to distinct one-element sets $\PPP_1,\PPP_2,\ldots$
in Protocol~2.
The argument in the proof of Lemma~\ref{lem:levy-general}
(which is due to a referee)
demonstrates that Protocols~1 and~2 are essentially equivalent.

\ifnotJOURNAL
\section{B\'artfai and R\'ev\'esz's zero-one law}
\label{sec:bartfai-revesz}

In this section we will illustrate L\'evy's zero-one law
by deducing a simple game-theoretic analogue
of a zero-one law \cite{bartfai/revesz:1967} for dependent random variables.
Intuitively, the role of the sample space will now be played
by the set $\XXX^{\infty}$ of all moves by Reality,
and the role of situations will be played by elements of $\XXX^*$.
If $\chi=x_1x_2\ldots\in\XXX^{\infty}$,
we let $\chi_n$ stand for $x_n\in\XXX$ for $n\in\bbbn$,
and let $\chi^n$ stand for $x_1\ldots x_n\in\XXX^n$ for $n\in\{0,1,\ldots\}$.

A \emph{forecasting system} $\Phi$ is any function $\Phi:\XXX^*\to\PPP$
such that $\Phi(\chi^{n-1})\in\PPP_n$ for all $\chi\in\XXX^{\infty}$ and $n\in\bbbn$.
A forecasting system can serve as a strategy for Forecaster in Protocol~2,
giving Forecaster's move as function of Reality's moves.
For each $\chi\in\XXX^{\infty}$ define
$$
  \chi_{\Phi}
  :=
  \Phi(\Box) \chi_1 \Phi(\chi^1) \chi_2 \Phi(\chi^2) \chi_3\ldots
  \in
  \Omega.
$$
For each $E\subseteq\XXX^{\infty}$ define
$E_{\Phi}:=\{\chi_{\Phi}\st\chi\in E\}$.
For $E\subseteq\XXX^{\infty}$, $\chi\in\XXX^{\infty}$, and $n\in\{0,1,\ldots\}$,
set
$$
  \UpperProb_{\Phi}(E\givn\chi^n)
  :=
  \UpperProb(E_{\Phi}\givn(\chi_{\Phi})^n),
  \quad
  \LowerProb_{\Phi}(E\givn\chi^n)
  :=
  1-\UpperProb_{\Phi}(E^c\givn\chi^n).
$$
As before, ``${}\givn\Box$'' may be omitted,
so that $\UpperProb_{\Phi}(E)=\UpperProb(E_{\Phi})$
and $\LowerProb_{\Phi}(E)=\LowerProb(E_{\Phi})$.
An $E\subseteq\XXX^{\infty}$ holds \emph{$\Phi$-almost surely}
(\emph{$\Phi$-a.s.})\ if $\UpperProb_{\Phi}(E^c)=0$.

For each $N\in\bbbn$,
let $\FFF_N$ be the set of all $E\subseteq\XXX^{\infty}$
such that, for all $\chi,\chi'\in\Omega$,
$$
  \left(
    \chi\in E,
    \forall n\ge N: \chi'_n=\chi_n
  \right)
  \Longrightarrow
  \chi'\in E.
$$
Let us say that a forecasting system $\Phi$
is \emph{$\delta$-mixing}, for $\delta\in[0,1)$,
if there exists a function $a:\bbbn\to\bbbn$ such that
\begin{equation}\label{eq:mixing}
  \UpperProb_{\Phi}(E\givn\chi^n)
  -
  \UpperProb_{\Phi}(E)
  \le
  \delta
  \quad
  \text{$\Phi$-a.s.}
\end{equation}
for each $n\in\bbbn$ and each $E\in\FFF_{n+a(n)}$.
By a \emph{tail set} we mean an element of $\cap_N\FFF_N$.
Now we can state an approximate zero-one law,
which is a game-theoretic analogue of the main result of \cite{bartfai/revesz:1967}.
\begin{corollary}\label{cor:bartfai}
  Suppose $\EEE_p$ is a superexpectation for all $p\in\PPP$.
  Let $\delta\in[0,1)$
  and let $\Phi$ be a $\delta$-mixing forecasting system.
  If $E\subseteq\XXX^{\infty}$ is a tail set,
  then $\UpperProb_{\Phi}(E)=0$ or $\UpperProb_{\Phi}(E)\ge1-\delta$.
\end{corollary}
\begin{proof}
  Fix a tail set $E\subseteq\XXX^{\infty}$;
  (\ref{eq:mixing}) then holds for all $n$.
  By the last part of Lemma \ref{lem:referee-4}
  (which is also valid in Protocol~2),
  there is $A\subseteq\XXX^{\infty}$ such that $\UpperProb_{\Phi}(A)=0$
  and
  \begin{equation}\label{eq:bartfai-1}
    \UpperProb_{\Phi}(E\givn\chi^n)
    -
    \UpperProb_{\Phi}(E)
    \le
    \delta
  \end{equation}
  holds for all $n$ and all $\chi\notin A$.
  By definition, (\ref{eq:bartfai-1}) means
  \begin{equation}\label{eq:bartfai-2}
    \UpperProb(E_{\Phi}\givn(\chi_{\Phi})^n)
    -
    \UpperProb(E_{\Phi})
    \le
    \delta.
  \end{equation}
  By Corollary~\ref{cor:levy} and Lemma~\ref{lem:levy-general},
  there is a set $B\subseteq\Omega$ such that $\UpperProb(B)=0$ and
  \begin{equation}\label{eq:bartfai-3}
    \UpperProb(E_{\Phi}\givn\omega^n) \to 1
    \quad
    (n\to\infty)
  \end{equation}
  for all $\omega\in E_{\Phi}\setminus B$.
  Letting $n\to\infty$ in (\ref{eq:bartfai-2}) and using (\ref{eq:bartfai-3}),
  we can see that $\UpperProb(E_{\Phi})\ge1-\delta$
  (i.e., $\UpperProb_{\Phi}(E)\ge1-\delta$)
  for all $\chi$ such that $\chi\notin A$ (i.e., $\chi_{\Phi}\notin A_{\Phi}$)
  and $\chi_{\Phi}\in E_{\Phi}\setminus B$.

  Suppose $\UpperProb_{\Phi}(E)\ge1-\delta$ is violated.
  Then there are no $\chi$ satisfying
  $\chi_{\Phi}\notin A_{\Phi}$ and $\chi_{\Phi}\in E_{\Phi}\setminus B$.
  In other words,
  $E_{\Phi}\setminus B\subseteq A_{\Phi}$,
  which implies $E_{\Phi}\subseteq A_{\Phi}\cup B$,
  which in turn implies $\UpperProb(E_{\Phi})=0$,
  i.e., $\UpperProb_{\Phi}(E)=0$.
\end{proof}
Let us say that a set $E\subseteq\XXX^{\infty}$ is \emph{$\Phi$-unprobabilized}
if $\LowerProb_{\Phi}(E)<\UpperProb_{\Phi}(E)$.
An important special case of Corollary~\ref{cor:bartfai}
is the following zero-one law for ``weakly dependent'' trials
(cf.\ Corollary 1 in \cite{bartfai/revesz:1967}).
\begin{corollary}
  Suppose $\EEE_p$ is a superexpectation for all $p\in\PPP$.
  Let $\delta\in[0,1/2)$
  and let $\Phi$ be a $\delta$-mixing forecasting system.
  Every tail set $E\subseteq\XXX^{\infty}$
  satisfies $\LowerProb_{\Phi}(E)=1$,
  satisfies $\UpperProb_{\Phi}(E)=0$,
  or is $\Phi$-unprobabilized.
\end{corollary}
\begin{proof}
  It suffices to apply Corollary~\ref{cor:bartfai}
  to the tail sets $E$ and $E^c$.
\end{proof}

It is easy to strengthen Corollary~\ref{cor:bartfai}
by modifying the notion of a $\delta$-mixing forecasting system.
Let us say that the forecasting system is \emph{asymptotically $\delta$-mixing},
for $\delta\in[0,1)$,
if (\ref{eq:mixing}) holds for each $n\in\bbbn$
and each tail set $E$.
B\'artfai and R\'ev\'esz \cite{bartfai/revesz:1967}
do not introduce this notion (more precisely, its measure-theoretic version) explicitly,
but they do introduce two notions intermediate
between $\delta$-mixing and asymptotic $\delta$-mixing,
which they call stochastic $\delta$-mixing and $\delta$-mixing in mean.
The following proposition
is similar to (but much simpler than)
Theorems 2 and 3 in \cite{bartfai/revesz:1967}.
\begin{corollary}
  Suppose $\EEE_p$ is a superexpectation for all $p\in\PPP$.
  Let $\delta\in[0,1)$.
  The following two conditions are equivalent:
  \begin{enumerate}
  \item
    A forecasting system $\Phi$ is asymptotically $\delta$-mixing.
  \item
    Every tail set $E\subseteq\XXX^{\infty}$ satisfies
    $\UpperProb_{\Phi}(E)=0$ or $\UpperProb_{\Phi}(E)\ge1-\delta$.
  \end{enumerate}
\end{corollary}
\begin{proof}
  The argument of Corollary~\ref{cor:bartfai} shows
  that the first condition implies the second.
  Let us now assume the second condition
  and deduce the first.
  Let $n\in\bbbn$ and $E$ be a tail set.
  If $\UpperProb_{\Phi}(E)=0$,
  then $\UpperProb_{\Phi}(E\givn\chi^n)=0$ $\Phi$-a.s.\ can be proved
  similarly to the proof of Lemma \ref{lem:determinate},
  and so (\ref{eq:mixing}) holds.
  If $\UpperProb_{\Phi}(E)\ge1-\delta$,
  (\ref{eq:mixing}) is vacuous.
\end{proof}
\fi

\subsection*{Acknowledgements}

Our thinking about L\'evy's zero-one law
was influenced by a preliminary draft of \cite{bru/eid:2009}.
We are grateful to Gert de Cooman
for his questions that inspired some of the results
in Section \ref{sec:equivalent} of this article\ifnotJOURNAL,
  and to an anonymous referee of \cite{\Takemura}
  who pointed out to us the zero-one law in \cite{bartfai/revesz:1967}.\fi
This article has benefitted very much from a close reading by its anonymous referee,
whose penetrating comments have led to a greatly improved presentation
(in particular, Lemmas~\ref{lem:referee-1}, \ref{lem:referee-3},
\ref{lem:referee-2}, \ref{lem:referee-4}
and the final statements of Theorem~\ref{thm:levy} and Lemma~\ref{lem:martingale}
are due to him or her)
and helped us correct a vacuous statement in a previous version.
Andrzej Ruszczy\'nski has brought to our attention
the literature on coherent measures of risk.
Our work has been supported in part by EPSRC grant EP/F002998/1.

\end{document}